\documentclass[11pt,a4paper]{article}
\usepackage{amsfonts,amsgen,amstext,amsbsy,amsopn,amsfonts,amssymb,amscd}
\usepackage[leqno]{amsmath}
\usepackage[amsmath,amsthm,thmmarks]{ntheorem}
\usepackage{dsfont}
\usepackage{epsf,epsfig}
\usepackage{float}
\usepackage{ebezier,eepic}
\usepackage{color}
\usepackage{tikz}
\usepackage{multirow}
\usepackage{mathrsfs}
\usepackage{graphicx}
\usepackage{subfigure}
\newtheorem{thm}{Theorem}[section]

\newtheorem{prop}[thm]{Proposition}

\newtheorem{example}[thm]{Example}
\newtheorem{lem}[thm]{Lemma}
\newtheorem{false statement}{False statement}
\newtheorem{cor}[thm]{Corollary}
\newtheorem{fact}[thm]{Fact}

\theoremstyle{definition}
\newtheorem{defn}[thm]{Definition}
\newtheorem{claim}{Claim}

\makeatletter \@addtoreset{equation}{section}

\def\hh{\mathcal{H}}

\def\hht{\mathcal{T}}
\def\hf{\mathcal{F}}
\def\hg{\mathcal{G}}

\def\ha{\mathcal{A}}
\def\hb{\mathcal{B}}

\def\he{\mathcal{E}}
\def\hs{\mathcal{S}}
\def\hr{\mathcal{R}}

\def\hj{\mathcal{J}}
\def\hl{\mathcal{L}}
\def\hp{\mathcal{P}}

\begin{document}
\title{\bf\Large Improved bounds concerning the maximum degree of intersecting hypergraphs}
\date{}
\author{Peter Frankl$^1$, Jian Wang$^2$\\[10pt]
$^{1}$R\'{e}nyi Institute, Budapest, Hungary\\[6pt]
$^{2}$Department of Mathematics\\
Taiyuan University of Technology\\
Taiyuan 030024, P. R. China\\[6pt]
E-mail:  $^1$frankl.peter@renyi.hu, $^2$wangjian01@tyut.edu.cn
}
\maketitle
\begin{abstract}
For positive integers $n>k>t$ let $\binom{[n]}{k}$ denote the collection of all $k$-subsets of the standard $n$-element set $[n]=\{1,\ldots,n\}$. Subsets of $\binom{[n]}{k}$ are called $k$-graphs. A $k$-graph $\hf$ is called $t$-intersecting if $|F\cap F'|\geq t$ for all $F,F'\in \hf$. One of the central results of extremal set theory is the Erd\H{o}s-Ko-Rado Theorem which states that for $n\geq (k-t+1)(t+1)$ no $t$-intersecting $k$-graph has more than $\binom{n-t}{k-t}$ edges. For $n$ greater than this threshold the $t$-star (all $k$-sets containing a fixed $t$-set) is the only family attaining this bound. Define $\hf(i)=\{F\setminus \{i\}\colon i\in F\in \hf\}$. The quantity $\varrho(\hf)=\max\limits_{1\leq i\leq n}|\hf(i)|/|\hf|$ measures how close a $k$-graph is to a star. The main result (Theorem 1.5) shows that $\varrho(\hf)>1/d$ holds if $\hf$ is 1-intersecting, $|\hf|>2^dd^{2d+1}\binom{n-d-1}{k-d-1}$  and $n\geq 4(d-1)dk$. Such a statement can be deduced from the results of \cite{F78-2} and \cite{DF}, however only for much larger values of $n/k$ and/or $n$. The proof is purely combinatorial, it is based on a new method: shifting ad extremis. The same method is applied to obtain some nearly optimal bounds in the case of $t\geq 2$ (Theorem 1.11) along with a number of related results.
\end{abstract}

\section{Introduction}

For positive integers $n\geq k$ let $[n]=\{1,\ldots,n\}$ be the standard $n$-element set and $\binom{[n]}{k}$ the collection of its $k$-subsets. A family $\hf \subset \binom{[n]}{k}$ is called {\it $t$-intersecting} if $|F\cap F'|\geq t$ for all $F,F'\in \hf$ and $t$ a positive integer. In the case $t=1$ we usually omit $t$ and speak of intersecting families. Let us recall one of the fundamental results of extremal set theory.

\begin{thm}[Exact Erd\H{o}s-Ko-Rado Theorem (\cite{ekr}, \cite{F78}, \cite{W3})] Let $k\geq t>0$, $n\geq n_0(k,t)=(k-t+1)(t+1)$. Suppose that $\hf\subset \binom{[n]}{k}$ is $t$-intersecting. Then
\begin{align}\label{ineq-ekr}
|\hf|\leq \binom{n-t}{k-t}.
\end{align}
\end{thm}

Let us note that $|\hs(n,k,t)|=\binom{n-t}{k-t}$ holds for the {\it full star}
\[
\hs(n,k,t) = \left\{S\in \binom{[n]}{k}\colon [t]\subset S\right\}
\]
 and for $n>n_0(k,t)$ up to isomorphism $\hs(n,k,t)$ is the only family to achieve equality in \eqref{ineq-ekr}. The exact bound $n_0(k,t)=(k-t+1)(t+1)$ is due to Erd\H{o}s, Ko and Rado in the case $t=1$. For $t\geq 15$ it was established in \cite{F78}. Wilson \cite{W3} closed the gap $2\leq t\leq 14$ by a proof valid for all $t\geq 1$.

 Let us recall some standard notation. Set $\cap \hf=\cap\{F\colon F\in \hf\}$. If $|\cap \hf|\geq t$ then $\hf$ is called a {\it $t$-star}, for $t=1$ we usually omit the 1.

 For a subset $E\subset [n]$ and a family $\hf\subset \binom{[n]}{k}$ define
 \[
\hf(E) =\{F\setminus E\colon E\subset F\in\hf\}, \ \hf(\overline{E})=\{F\in \hf\colon F\cap E=\emptyset\}.
 \]
In the case $E=\{i\}$ we simply use $\hf(i)$ and $\hf(\bar{i})$ to denote $\hf(\{i\})$ and $\hf(\overline{\{i\}})$, respectively. In analogy,
\[
\hf(u,v,\bar{w}):= \left\{F\setminus\{u,v\}\colon F\in \hf, F\cap \{u,v,w\}=\{u,v\}\right\}.
\]

Let $(x_1,\ldots,x_k)$ denote the set $\{x_1,\ldots,x_k\}$ where we know or want to stress that $x_1<\ldots<x_k$.

Let us define the {\it shifting partial order} $\prec$ where $P\prec Q$ for $P=(x_1,\ldots,x_k)$  and $Q=(y_1,\ldots,y_k)$ iff $x_i\leq y_i$ for all $1\leq i\leq k$. This partial order can be traced back to \cite{ekr}.

A family $\hf\subset \binom{[n]}{k}$ is called {\it initial} if $F\prec G$, $G\in \hf$ always implies $F\in \hf$. Note that $|\hf(1)|\geq |\hf(2)|\geq \ldots \geq |\hf(n)|$ for an initial family. Initial families have some nice properties. Let us give here one simple example.

\begin{prop}\label{prop-1.2}
Suppose that $\hf\subset 2^{[n]}$ is initial and $t$-intersecting. Then $\hf(\overline{[s]})$ is $(t+s)$-intersecting.
\end{prop}
\begin{proof}
Suppose for contradiction that $F,F'\in \hf(\overline{[s]})$, $E:= F\cap F'$ satisfies $|E|=t+j$, $j+1\leq s$. Fix $E_0\subset E$, $|E_0|=j+1$. Then $F'':=[j+1]\cup F'\setminus E_0$ satisfies $F''\prec F'$ whence $F''\in \hf$. However $F\cap F''=E\setminus E_0$ and $|E\setminus E_0|=t-1$, the desired contradiction.
\end{proof}

Let us define the quantity
\[
\varrho(\hf) =\max\left\{\frac{|\hf(i)|}{|\hf|}\colon 1\leq i\leq n\right\}.
\]
Since $\varrho(\hf)=1$ if and only if $\hf$ is a star, in a way $\varrho(\hf)$ measures how far a family is from a star.

A set $T$ is called a {\it $t$-transversal} of $\hf$ if $|T\cap F|\geq t$ for all $F\in \hf$. If $\hf$ is $t$-intersecting then each $F\in\hf$ is a $t$-transversal. Define
\[
\tau_t(\hf) =\min\{|T|\colon T \mbox{ is a $t$-transversal of }\hf\}.
\]
For $t=1$ we usually omit the 1.

\begin{prop}\label{prop-1.3}
If $\hf$ is $t$-intersecting, then
\begin{align}\label{ineq-1.2}
\varrho(\hf) \geq \frac{t}{\tau_t(\hf)}.
\end{align}
\end{prop}

\begin{proof}
Fix a $t$-transversal $T$ of $\hf$ with $|T| = \tau_t(\hf)$. Then
\[
t|\hf| \leq \sum_{i\in T} |\hf(i)|\leq |T| \cdot \max\{|\hf(i)|\colon i\in T\},
\]
implying \eqref{ineq-1.2}.
\end{proof}

Obviously, $\tau_t(\hf)=t$ iff $\hf$ is a $t$-star.

\begin{example}
For $n>k>t>0$ define
\[
\ha(n,k,t) =\left\{A\in \binom{[n]}{k}\colon |A\cap [t+2]|\geq t+1\right\}.
\]
\end{example}

Clearly, $\ha=\ha(n,k,t)$ is $t$-intersecting, $\varrho(\ha) =\frac{t+1+o(1)}{t+2}$, $\tau_t(\hf) =t+1$. We should note that for $2k-t<n<(k-t+1)(t+1)$, $|\ha|> \binom{n-t}{k-t}$ with equality for $n=(k-t+1)(t+1)$.

In \cite{F78-2} it was shown that for any positive $\varepsilon$ and $n>n_1(k,t,\varepsilon)$, $\varrho(\hf)<1-\varepsilon$ implies $|\hf|\leq |\ha|$ for any $t$-intersecting family $\hf\subset \binom{[n]}{k}$.

The value of $n_1(k,t,\varepsilon)$ is implicit in \cite{F78-2}. With careful calculation (cf. e.g. \cite{FW2022-2}) for fixed $\varepsilon >0$ one can prove a bound quadratic in $k$. Dinur and Friedgut \cite{DF} introduced the so-called junta-method that leads to strong results for $n>ck$, however the value of the constant is astronomically large and it is further dependent on the particular problem (the same is true for the recent advances of Keller and Lifschitz \cite{KL}).

The aim of the present paper is to prove some similar results concerning $\varrho(\hf)$ for $t$-intersecting families for $n>ck$ with relatively small constants $c$. Let us state here our main result for the case $t=1$.

\begin{thm}\label{main-1}
Let $n,k,d$ be  integers, $k> d\geq 2$, $n\geq 4(d-1)dk$.
If $\hf\subset \binom{[n]}{k}$ is intersecting and  $|\hf|\geq 2^dd^{2d+1}\binom{n-d-1}{k-d-1}$, then $\varrho(\hf)> \frac{1}{d}$.
\end{thm}

Let us stress once more that $\varrho(\hf)>\frac{1}{d}$ follows from the results of \cite{F78-2} and \cite{DF} however only for much larger value of $n$.

On the other hand for initial families one can deduce $\varrho(\hf)>\frac{1}{d}$ under much milder constraints (cf. \cite{FKupavskii}). The problem is that one cannot transform a general family into an initial family without increasing $\varrho(\hf)$. To circumvent this difficulty we are going to apply the recently developed method of shifting ad extremis that we will explain in Section 2.

Two families $\hf$, $\hg$ are called cross-intersecting if $F\cap G\neq\emptyset$ for all $F\in \hf$, $G\in \hg$. Our next result is for cross-intersecting families.

\begin{thm}\label{thm-main3}
Suppose that $n\geq 39k$, $\hf,\hg\subset \binom{[n]}{k}$ are cross-intersecting with $|\hf|,|\hg|\geq 2\binom{n-2}{k-2}+4\binom{n-3}{k-3}$. Then $\min\{\varrho(\hf),\varrho(\hg)\}>\frac{1}{2}+\frac{k-2}{2(n-2)}$.
\end{thm}

\begin{example}\label{examp-1}
Let $n\geq 2k$. Define
\[
\hp=\left\{P\in \binom{[n]}{k}\colon (1,2)\subset P \mbox{ or }(3,4)\subset P\right\},\ \hr=\left\{R\in \binom{[n]}{k}\colon (1,3)\subset R \mbox{ or }(2,4)\subset R\right\},
\]
and
\[
\hs=\left\{S\in \binom{[n]}{k}\colon (1,4,5)\subset S \mbox{ or }(2,3,5)\subset S\right\}.
\]
Clearly,  $\hp\cup \hs,\hr\cup \hs$ are cross-intersecting. Simple computation shows that
\begin{align*}
|\hp\cup \hs|=|\hr\cup \hs|>2\binom{n-2}{k-2}+2\binom{n-3}{k-3}-5\binom{n-4}{k-4}
\end{align*}
and
\[
\varrho(\hp\cup \hs)=\varrho(\hr\cup\hs)<\frac{1}{2}+\frac{(k-2)^2}{(n-2)^2}<\frac{1}{2}+\frac{k-2}{2(n-2)}.
\]
This shows that $4{n-3 \choose k-3}$ in Theorem \ref{thm-main3} cannot be replaced by $c{n-3 \choose k-3}$ for $c<2$.
\end{example}

\begin{example}\label{examp-1}
Set $\ha=\{(1,2),(3,4)\}$, $\hb=\{(1,3),(1,4),(2,3),(2,4)\}$. Define
\[
\hf_\ha=\left\{F\in \binom{[n]}{k}\colon (1,2)\subset F \mbox{ or }(3,4)\subset F\right\}, \
\hg_\hb=\left\{G\in \binom{[n]}{k}\colon \exists B\in \hb, B\subset G\right\}.
\]
Note that $\hf_\ha,\hg_\hb$ are cross-intersecting. Simple computation shows that $\varrho(\hf_\ha) <\frac{1}{2}+\frac{k-2}{2(n-2)}, \varrho(\hg_\hb)<\frac{1}{2}+\frac{k-2}{2(n-2)}$ and
\[
|\hf_\ha|=2\binom{n-2}{k-2}-\binom{n-4}{k-4}, \ |\hg_\hb|=4\binom{n-4}{k-2}+4\binom{n-4}{k-3}+\binom{n-4}{k-4}.
\]
This shows in a strong sense that it is not sufficient to have one of the two families being considerably larger than the bound in Theorem \ref{thm-main3}.
\end{example}

Let $\hht(n,k)$ denote the triangle family defined by
\[
\hht(n,k) =\left\{F\in \binom{[n]}{k}\colon |F\cap[3]|\geq 2\right\}.
\]
Note that $\hht(n,k)$ is initial and $\varrho(\hht(n,k))=\frac{2}{3}+o(1)$.

\begin{thm}\label{thm-main4}
Suppose that $n\geq 39k$, $\hf,\hg\subset \binom{[n]}{k}$ are cross-intersecting with $|\hf|,|\hg|\geq |\hht(n,k)|$. Then $\min\{\varrho(\hf),\varrho(\hg)\}>\frac{2}{3}\left(1-\frac{k-2}{n-2}\right)$.
\end{thm}

For $\min\{|\hf|,|\hg|\}=\Omega\left(\binom{n-3}{k-3}\right)$, we prove the following result.

\begin{thm}\label{thm-main2}
Suppose that $0<\varepsilon\leq \frac{1}{58}$, $\hf,\hg\subset \binom{[n]}{k}$ are cross-intersecting with $|\hf|,|\hg|\geq \frac{1}{\varepsilon}\binom{n-3}{k-3}$. Then $\max\{\varrho(\hf),\varrho(\hg)\}>\frac{1}{2}-\varepsilon$.
\end{thm}

For $t$-intersecting families, we obtain the following result.

\begin{thm}\label{thm-main5}
Let $\hf\subset \binom{[n]}{k}$ be a $t$-intersecting family with $t\geq 2$. If $|\hf|> (t+1)\binom{n-1}{k-t-1}$ and $n\geq 2t(t+2)k$,
then $\varrho(\hf)> \frac{t}{t+1}$.
\end{thm}

Let us present some results that are needed in our proofs.  Define the {\it lexicographic order} $A <_{L} B$ for $A,B\in \binom{[n]}{k}$ by $A<_L B$ iff $\min\{i\colon i\in A\setminus B\}<\min\{i\colon i\in B\setminus A\}$. E.g., $(1,2,9)<_L(1,3,4)$. For $n>k>0$ and $\binom{n}{k}\geq m>0$ let $\hl(n,k,m)$ denote the first $m$ sets $A\in \binom{[n]}{k}$ in the lexicographic order. For  $X\subset [n]$ with $|X|>k>0$ and $\binom{|X|}{k}\geq m>0$, we also use  $\hl(X,k,m)$ to denote the first $m$ sets $A\in \binom{X}{k}$ in the lexicographic order.

A powerful tool is the  Kruskal-Katona Theorem (\cite{Kruskal,Katona66}), especially its reformulation due to Hilton \cite{Hilton}.

{\noindent\bf Hilton's Lemma (\cite{Hilton}).} Let $n,a,b$ be positive integers, $n\geq a+b$. Suppose that $\ha\subset \binom{[n]}{a}$ and $\hb\subset \binom{[n]}{b}$ are cross-intersecting. Then $\hl(n,a,|\ha|)$ and $\hl(n,b,|\hb|)$ are cross-intersecting as well.

For $\hf\subset \binom{[n]}{k}$ define the $\ell$th shadow of $\hf$,
\[
\partial^{\ell}\hf =\left\{G\colon |G|=k-\ell, \exists F\in \hf, G\subset F\right\}.
\]
For $\ell=1$ we often omit the superscript.

The Katona Intersecting Shadow Theorem gives an inequality concerning the sizes of  a $t$-intersecting family and its shadow.

\vspace{6pt}
{\noindent\bf Katona Intersecting Shadow Theorem (\cite{Katona}).} Suppose that $n\geq 2k-t$, $t\geq \ell \geq 1$. Let $\emptyset \neq \ha \subset \binom{[n]}{k}$ be a $t$-intersecting family. Then
\begin{align}\label{ineq-ishadow}
|\partial^{\ell} \ha| \geq |\ha| \frac{\binom{2k-t}{k-\ell}}{\binom{2k-t}{k}}
\end{align}
with equality holding if and only if $\hf$ is isomorphic to $\binom{[2k-t]}{k}$.
\vspace{6pt}

 A generalized version of the Katona Intersecting Shadow Theorem was proved in \cite{F76}.

\begin{thm}[\cite{F76}]
Let $k_1,k_2,\ell_1,\ell_2, t$ be integers satisfying $1\leq \ell_1<k_1$ and $1\leq \ell_2<k_2$. Let $\emptyset\neq \ha \subset \binom{[n]}{k_1}$, $\emptyset\neq \hb \subset \binom{[n]}{k_2}$ and $\ha,\hb$ are cross $t$-intersecting. Then
\begin{align}\label{ineq-gshadow}
\mbox{either }|\partial^{\ell_1} \ha| \geq |\ha|\frac{\binom{2k_1-t}{k_1-\ell_1}}{\binom{2k_1-t}{k_1}} \mbox{ or }|\partial^{\ell_2} \hb| \geq |\hb|\frac{\binom{2k_2-t}{k_2-\ell_2}}{\binom{2k_2-t}{k_2}}.
\end{align}
\end{thm}

We need a notion called pseudo $t$-intersecting, which was introduced in \cite{FK2021}.
A family $\hf\subset \binom{[n]}{k}$ is said to be {\it pseudo $t$-intersecting} if for every  $F\in \hf$ there exists $0\leq \ell\leq k-t$ such that $|F\cap [2\ell+t]|\geq \ell+t$. It is proved in \cite{F91} that  \eqref{ineq-ishadow} holds for pseudo $t$-intersecting families as well.

\begin{thm}[\cite{F78}]
Let $\hf\subset \binom{[n]}{k}$ be an initial family with $0\leq t<k$. If  $\hf$ is pseudo $t$-intersecting, then
\begin{align}\label{ineq-walk}
|\hf| \leq \binom{n}{k-t}.
\end{align}
\end{thm}

\begin{prop}[\cite{FK2021}]\label{fk2021}
If $\hf\subset \binom{[n]}{k},\hg\subset \binom{[n]}{\ell}$ are cross $t$-intersecting, then either both $\hf$ and $\hg$ are pseudo $t$-intersecting, or one of them is pseudo $(t+1)$-intersecting.
\end{prop}

The following inequalities for  cross $t$-intersecting families can be deduced from Proposition \ref{fk2021}.

\begin{cor}[\cite{F78}]\label{thm-tintersecting}
Suppose that $\ha,\hb\subset \binom{[n]}{k}$ are cross $t$-intersecting, $|\ha|\leq |\hb|$. Then either
\begin{align}
|\hb|\leq \binom{n}{k-t} \mbox{ \rm or }\label{ineq-tinterhb}\\[5pt]
|\ha| \leq \binom{n}{k-t-1}.\label{ineq-tinterha}
\end{align}
\end{cor}

The next two results are consequences of Hilton's Lemma.

\begin{lem}[\cite{FW2022}]\label{lem-1}
Let $\hf,\hg\subset \binom{[n]}{k}$ be cross-intersecting. Suppose that for some $\{x,y\}\in \binom{[n]}{2}$, $|\hf(x,y)|\geq \binom{n-3}{k-3}+\binom{n-4}{k-3}+\binom{n-6}{k-4}$. Then
\begin{align}
|\hg(\bar{x},\bar{y})|\leq \binom{n-5}{k-3}+\binom{n-6}{k-3}.\label{ineq-hgxy}
\end{align}
\end{lem}

\begin{lem}[\cite{FW2022}]
Suppose that $\hf,\hg\subset \binom{[n]}{k}$ are cross-intersecting and for some $\{x,y\}\in \binom{[n]}{2}$, $|\hf(x,y)|\geq \binom{n-3}{k-3}+\binom{n-4}{k-3}+\binom{n-5}{k-3}+\binom{n-7}{k-4}$. Then
\begin{align}
|\hg(\bar{x},\bar{y})|\leq \binom{n-6}{k-4}+\binom{n-7}{k-4}.\label{ineq-hgxy2}
\end{align}
\end{lem}

For $\hf\subset \binom{[n]}{k}$ and $P\in \binom{[n]}{2}$, define
\[
\hf_P=\{F\in \hf\colon F\cap P\neq\emptyset\}.
\]

\begin{lem}[\cite{FW2022}]\label{lem-3.5}
Let  $\hf\subset \binom{[n]}{k}$ be an intersecting family with $|\hf(P)|\leq M$ for every $P\in \binom{[n]}{2}$ and $M\geq 2\binom{n-5}{k-3}$. Suppose that $R,Q\in \binom{[n]}{2}$ and  $R\cap Q=\emptyset$. Then
\begin{align}\label{ineq-FQR-general}
|\hf_R\cap \hf_Q|\leq 3M+\binom{n-7}{k-5}+\binom{n-8}{k-5}.
\end{align}
\end{lem}

We need the following inequalities concerning binomial coefficients.

\begin{prop}
Let $n,k,i$ be positive integers. Then
\begin{align} \label{ineq-key1}
&\binom{n-i}{k} \geq \frac{n-ik}{n}\binom{n}{k}, \mbox{\rm \ for } n>ik.
\end{align}
\end{prop}
\begin{proof}
It is easy to check for all $b>a>0$ that
\begin{align}\label{ineq-prekey}
ba>(b+1)(a-1) \mbox{ holds.}
\end{align}
Note that
\[
\frac{\binom{n-i}{k}}{\binom{n}{k}} = \frac{(n-k)(n-k-1)\cdots(n-k-(i-1))}{n(n-1)\cdots(n-(i-1))}.
\]
Applying \eqref{ineq-prekey} repeatedly we see that the enumerator is greater  than $(n-1)(n-2)\cdots(n-(i-1))(n-ki)$ implying
\[
\frac{\binom{n-i}{k}}{\binom{n}{k}} \geq 1-\frac{ik}{n}.
\]
Thus \eqref{ineq-key1} holds.
\end{proof}

\begin{cor}
Let $n,k,t$ be positive integers. If $n\geq 2(t-1)(k-t)$ and $k>t\geq 2$, then
\begin{align} \label{ineq-key11}
&\binom{n-t-2}{k-t-2} \geq \frac{1}{2}\binom{n-3}{k-t-2}.
\end{align}
\end{cor}

\begin{proof}
Note that
\[
n\geq 2(t-1)(k-t)= 2(t-1)(k-t-2)+4(t-1)> 2(t-1)(k-t-2)+3.
\]
By \eqref{ineq-key1} we have
\begin{align*}
\binom{n-t-2}{k-t-2} &\geq \frac{(n-3)-(t-1)(k-t-2)}{n-3}\binom{n-3}{k-t-2}\geq \frac{1}{2} \binom{n-3}{k-t-2}.
\end{align*}
\end{proof}

Let us state one more result.

\begin{prop}[\cite{F78-2}]
Let $\ha,\hb\subset \binom{[n]}{k}$ be non-empty cross-intersecting and $n\geq 2k$. If $|\ha|\geq |\hb|\geq  \binom{n-2}{k-2}$, then
\begin{align}\label{F78}
|\ha|+|\hb|\leq 2\binom{n-1}{k-1}.
\end{align}
\end{prop}

\section{Shifting ad extremis and the proof of Theorem \ref{main-1}}

Let us recall an important operation called shifting introduced by Erd\H{o}s, Ko and Rado \cite{ekr}. For $\hf\subset{[n]\choose k}$ and $1\leq i< j\leq n$, define
$$S_{ij}(\hf)=\left\{S_{ij}(F)\colon F\in\hf\right\},$$
where
$$S_{ij}(F)=\left\{
                \begin{array}{ll}
                  (F\setminus\{j\})\cup\{i\}, & j\in F, i\notin F \text{ and } (F\setminus\{j\})\cup\{i\}\notin \hf; \\[5pt]
                  F, & \hbox{otherwise.}
                \end{array}
              \right.
$$

Recall that $\ha\subset \binom{[n]}{k}$ is called initial iff for all  $A,B\in \binom{[n]}{k}$, $A\prec B$ and $B\in \ha$ imply $A\in \ha$.

The following statement goes back to Katona \cite{Katona66}. Let us include the very short proof.

\begin{prop}
Let $\hf\subset \binom{[n]}{k}$ be an initial family. Then
\begin{align}\label{ineq-shifted}
\partial \hf(\bar{1})\subset \hf(1).
\end{align}
\end{prop}

\begin{proof}
Indeed, if $E\subset F\in \hf(\bar{1})$ and $E=F\setminus \{j\}$. Then by initiality $E\cup\{1\}\in \hf$, i.e., $E\in \hf(1)$.
\end{proof}

Define the {\it matching number $\nu(\hf)$} of $\hf$ as the maximum number of pairwise disjoint edges in $\hf$. Note that $\nu(\hf)=1$ iff $\hf$ is intersecting.
\begin{cor}
Let $\hf\subset \binom{[n]}{k}$ be an initial family. If $\hf\subset \binom{[n]}{k}$ has matching number $s$, then $\varrho(\hf)\geq \frac{1}{s+1}$.
\end{cor}

\begin{proof}
It is proved in \cite{F91} that if $\hg$ has matching number $s$ then its shadow has size at least $|\hg|/s$. By \eqref{ineq-shifted} we infer
\[
|\hf(1)|\geq |\partial \hf(\bar{1})|\geq |\hf(\bar{1})|/s
\]
whence $\varrho(\hf)\geq \frac{1}{s+1}$.
\end{proof}

Let us define formally the notion of {\it shifting ad extremis} developed recently (cf. \cite{FW2022}). It can be applied  to one, two  or several families. For notational convenience we explain it for the case of two families in detail.

Let $\hf\subset \binom{[n]}{k}$, $\hg\subset \binom{[n]}{\ell}$ be two families and suppose that we are concerned, as usual in extremal set theory, to obtain upper bounds for $|\hf|+|\hg|$, $|\hf||\hg|$ or some other function $f$ of $|\hf|$ and $|\hg|$. For this we suppose that $\hf$ and $\hg$ have certain properties (e.g., cross-intersecting and non-trivial). Since $|S_{ij}(\hh)|=|\hh|$ for all families $\hh$, it is convenient to apply $S_{ij}$ simultaneously to $\hf$ and $\hg$. Certain properties, e.g., $t$-intersecting, cross-intersecting or $\nu(\hf)\leq r$ are known to be maintained by $S_{ij}$. However, some other properties may be destroyed, e.g., non-triviality, $\varrho(\hg)\leq c$, etc. Let $\hp$ be the collection of the latter properties that we want to maintain.

For any family $\hh$, define the quantity
\[
w(\hh) =\sum_{H\in \hh} \sum_{i\in H} i.
\]
Obviously $w(S_{ij}(\hh))\leq w(\hh)$ for $1\leq i<j\leq n$ with strict inequality unless $S_{ij}(\hh)=\hh$.

\begin{defn}\label{defn-2.1}
Suppose that $\hf\subset \binom{[n]}{k}$, $\hg\subset \binom{[n]}{\ell}$ are families having property $\hp$. We say that $\hf$ and $\hg$ have been {\it shifted ad extremis} with respect to $\hp$ if $S_{ij}(\hf)=\hf$ and $S_{ij}(\hg)=\hg$ for every pair $1\leq i<j\leq n$ whenever $S_{ij}(\hf)$ and $S_{ij}(\hg)$ also have property $\hp$.
\end{defn}

Let us show that we can obtain shifted ad extremis families by the following shifting ad extremis process. Let $\hf$, $\hg$ be cross-intersecting families with property $\hp$. Apply the shifting operation $S_{ij}$, $1\leq i<j\leq n$, to $\hf,\hg$ simultaneously and continue as long as the property $\hp$ is maintained. By abuse of notation, we keep denoting the current families by $\hf$ and $\hg$ during the shifting process. If $S_{ij}(\hf)$ or $S_{ij}(\hg)$ does not have property $\hp$, then we do not apply  $S_{ij}$ and choose a different pair $(i',j')$. However we keep returning to previously failed pairs $(i,j)$, because it might happen that at a later stage in the process $S_{ij}$ does not destroy property $\hp$ any longer. Note that the quantity $w(\hf)+w(\hg)$ is a positive integer and it decreases strictly in  each step. This guarantees that eventually we shall arrive at families that are shifted ad extremis with respect to $\hp$.

Let $\hf$, $\hg$ be shifted ad extremis families. A pair $(i,j)$ is called {\it shift-resistant} if either $S_{ij}(\hf)\neq \hf$ or $S_{ij}(\hg)\neq \hg$.

In the case of several families, $\hf_i\subset\binom{[n]}{k_i}$, $1\leq i\leq r$. It is essentially the same. One important property that is maintained by simultaneous shifting is {\it overlapping}, namely the non-existence of pairwise disjoint edges $F_1\in \hf_1,\ldots,F_r\in \hf_r$ (cf. \cite{HLS}).

\begin{proof}[Proof of Theorem \ref{main-1}]
Let $\hf\subset \binom{[n]}{k}$ be intersecting,  $|\hf|\geq 2^dd^{2d+1}\binom{n-d-1}{k-d-1}$ and $\varrho(\hf)\leq  \frac{1}{d}$. Without loss of generality, we may assume that  $\hf$ is shifted ad extremis for $\varrho(\hf)\leq \frac{1}{d}$. Then $S_{ij}(\hf)\neq \hf$ implies $\varrho(S_{ij}(\hf))>  \frac{1}{d}$. Thus, a pair $P=(i,j)$ is shift-resistant iff $|\hf(i)|+|\hf(j)|>|\hf|/d$.

Let $P_1,\ldots, P_s$ be a maximal collection of pairwise disjoint shift-resistant pairs, $P_i=(x_i,y_i)$, $1\leq i\leq s$. Clearly,
\begin{align}\label{ineq-new2.1}
\sum_{1\leq i\leq s} (|\hf(x_i)|+|\hf(y_i)|) \geq \frac{s}{d}|\hf|.
\end{align}

For a pair of subsets $E_0\subset E$, let us use the notation
\[
\hf(E_0,E) =\{F\setminus E\colon F\in \hf,\ F\cap E=E_0\}.
\]
Note that $\hf(E,E)=\hf(E)$ and $\hf(\emptyset,E)=\hf(\overline{E})$.
\begin{claim}
For all $D\in \binom{[n]}{d}$,
\begin{align}\label{ineq-new2.2}
|\hf(\overline{D})|\geq (d-1)|\hf(D)|.
\end{align}
\end{claim}
\begin{proof}
For any subset $E\subset [n]$ note the identity
\begin{align}\label{ineq-new2.3}
\sum_{x\in E}|\hf(x)|&= \sum_{1\leq j\leq |E|}\sum_{E_j\in \binom{E}{j}} j|\hf(E_j,E)|\nonumber\\[5pt]
&\geq  \sum_{E'\subset E, |E'|\geq 1} |\hf(E',E)|+(|E|-1)|\hf(E)|\nonumber\\[5pt]
&=  |\hf|-|\hf(\overline{E})|+(|E|-1)|\hf(E)|.
\end{align}
If $|E|=d$, then $\varrho(\hf)\leq \frac{1}{d}$ implies that the LHS of \eqref{ineq-new2.3} is less than $|\hf|$. Comparing with the RHS yields \eqref{ineq-new2.2}.
\end{proof}

\begin{claim}
For all  $D\in \binom{[n]}{d}$,
\begin{align}\label{ineq-new2.4}
|\hf(D)| < d\binom{n-d-1}{k-d-1}.
\end{align}
\end{claim}
\begin{proof}
 For convenience assume that $D=[n-d+1,n]$. Then $\hf(D)\subset \binom{[n-d]}{k-d}$, $\hf(\overline{D})\subset \binom{[n-d]}{k}$ and $\hf(D), \hf(\overline{D})$ are cross-intersecting. If
\[
|\hf(D)| \geq d\binom{n-d-1}{k-d-1}>\sum_{1\leq j\leq d}\binom{n-d-j}{k-d-1} +\binom{n-2d-2}{k-d-2},
\]
then  $\hl(n-d,k-d,|\hf(D)|)$ contains
\[
\left\{A\in \binom{[n-d]}{k-d}\colon A\cap [d] \neq \emptyset\right\}\bigcup \left\{A\in \binom{[d+1,n-d]}{k-d}\colon \{d+1,d+2\}\subset A\right\} .
\]
By Hilton's Lemma, we have
\[
\hl(n-d,k,|\hf(\overline{D})|)\subset  \left\{B\in \binom{[n-d]}{k}\colon [d] \subset B\mbox{ and }B\cap\{d+1,d+2\}\neq \emptyset\right\}.
\]
It follows that
\[
|\hf(\overline{D})|\leq \binom{n-2d-1}{k-d-1}+\binom{n-2d-2}{k-d-1}<|\hf(D)|,
\]
contradicting \eqref{ineq-new2.2}.
\end{proof}
\begin{claim}
\begin{align}\label{ineq-new2.5}
s\leq d^2-d.
\end{align}
\end{claim}
\begin{proof}
Assume that $s\geq d^2-d+1$. Define $E=P_1\cup\ldots\cup P_{d^2-d+1}$ and
\[
\hf_j =\{F\in \hf\colon |F\cap E|=j\}.
\]
Clearly $|E|=2(d^2-d+1)$. Then (as \eqref{ineq-new2.3}) the following formula is evident.
\begin{align*}
|\hf_j| =\sum_{E_j\in \binom{E}{j}}|\hf(E_j,E)|.
\end{align*}
By \eqref{ineq-new2.4} we have
\begin{align*}
\sum_{D\in \binom{E}{d}} |\hf(D)| <\binom{2(d^2-d+1)}{d}d\binom{n-d-1}{k-d-1}.
\end{align*}
Note that for $d\leq  j\leq |E|$, $E_j\in \binom{E}{j}$, a set $F\in \hf$ with $F\cap E=E_j$ is counted $\binom{j}{d}\geq j$ times in $\sum\limits_{D\in \binom{E}{d}} |\hf(D)|$ if $j>d$ and exactly once if $j=d$. Thus,
\begin{align}\label{ineq-new2.8}
|\hf_d|+\sum_{d< j\leq |E|} j|\hf_j|&=\sum_{D\in \binom{E}{d}} |\hf(D,E)|+\sum_{d<j\leq |E|}\sum_{E_j\in\binom{E}{j}} j|\hf(E_j,E)|\nonumber\\[5pt]
&\leq\sum_{d\leq j\leq |E|}\sum_{E_j\in\binom{E}{j}} \binom{j}{d}|\hf(E_j,E)|\nonumber\\[5pt]
&=\sum_{D\in \binom{E}{d}} |\hf(D)| <\binom{2(d^2-d+1)}{d}d\binom{n-d-1}{k-d-1}.
\end{align}
By \eqref{ineq-new2.1} and \eqref{ineq-new2.8}, we obtain that
\begin{align*}
\frac{d^2-d+1}{d} |\hf|
\leq \sum_{x\in E} |\hf(x)|&= \sum_{1\leq j\leq |E|} j|\hf_j| \\[5pt]
&< (d-1)\sum_{1\leq j\leq d}|\hf_j| + |\hf_d|+\sum_{d< j\leq |E|} j|\hf_j|\\[5pt]
&< (d-1)|\hf| +d\binom{2(d^2-d+1)}{d}\binom{n-d-1}{k-d-1}.
\end{align*}
It follows that
\[
|\hf| < d^2 \binom{2(d^2-d+1)}{d}\binom{n-d-1}{k-d-1}.
\]
Let $c(d) = d^2 \binom{2(d^2-d+1)}{d}$.  For $d\geq 4$, since $e^d<4^{d-1}\leq d^{d-1}$,  using $\binom{n}{k}<\left(\frac{en}{k}\right)^k$ we have
\[
c(d)<2^de^dd^{d+2}<2^dd^{2d+1},
\]
contradicting our assumption $|\hf|\geq2^dd^{2d+1}\binom{n-d-1}{k-d-1}$. For $d=2,3$, it can be checked directly that $c(d)<2^dd^{2d+1}$, contradicting our assumption as well.
\end{proof}

Fix $X\subset [n]$ with $|X|=2d^2-2d$ and $P_1\cup\ldots\cup P_s\subset X$. Define
\[
\hht = \left\{T\subset [n]\colon  |T|\leq d,\ |\hf(T)| > (2d^2)^{-|T|} |\hf|\right\}.
\]
By \eqref{ineq-new2.1}, there exists $x\in X$ such that
\[
|\hf(x)|\geq \frac{1}{2d}|\hf| > \frac{1}{2d^2} |\hf|,
\]
implying  $\hht\neq \emptyset$.
By \eqref{ineq-new2.4} we know that for every $D\in \binom{[n]}{d}$,
\[
|\hf(D)| <d \binom{n-d-1}{k-d-1} \leq (2d^2)^{-d} |\hf|.
\]
Thus we have $|T|\leq d-1$ for every $T\in \hht$.

Now choose $T\in \hht$ such that $|T|=t$ is maximum. Note that the maximality of $t$ implies that  for every $Z\subset [n]$ with $t<|Z|\leq d$
\begin{align}\label{ineq-hfz}
|\hf(Z)|< (2d^2)^{-|Z|} |\hf|.
\end{align}
Set $\ha=\hf(T,X\cup T)$ and  $U=[n]\setminus (X\cup T)$. Assume that
\[
U=\{u_1,u_2,\ldots,u_m\} \mbox{ with } u_1<u_2<\ldots <u_m.
\]
Let $Q=\{u_1,u_2,\ldots,u_{2d-t}\}$.  By  \eqref{ineq-hfz} we have
\begin{align}\label{ineq-2.2}
|\ha(\overline{Q})| &\geq |\hf(T)| - \sum_{x\in (X\setminus T)\cup Q} |\hf(T\cup \{x\})|\nonumber\\[5pt]
&\geq (2d^2)^{-t} |\hf| - (2d^2-2d+2d-t) (2d^2)^{-(t+1)} |\hf|\nonumber\\[5pt]
&= \frac{t}{(2d^2)^{t+1}}|\hf|>\binom{n-d-1}{k-d-1}.
\end{align}

\begin{claim}
For every $S\subset X\setminus T$,
\begin{align}\label{ineq-hfSEminusT}
|\hf(S,X\cup T)|\leq  2^{2d-1}\binom{n-d-1-|S|}{k-d-1-|S|}.
\end{align}
\end{claim}
\begin{proof}
Let $\hb=\hf(S,X\cup T)$. Since $P_1,P_2,\ldots,P_s$ is a maximal collection of pairwise disjoint shift-resistant pairs and $P_1\cup P_2\cup \ldots \cup P_s\subset X$, we infer that $\hf$ is initial on $[n]\setminus X$. It follows that $\ha,\hb$ are initial and cross-intersecting. For any $R\subset Q$ with $|R|=r\leq d$, we have $\ha(\overline{Q})\subset\binom{[n]\setminus (X\cup T\cup Q)}{k-t}$ and  $\hb(R,Q)\subset\binom{[n]\setminus (E\cup T\cup Q)}{k-|S|-r}$.
 By a similar argument as in the proof of Proposition \ref{prop-1.2},  we infer that $\ha(\overline{Q})$, $\hb(R,Q)$ are cross $(2d-t-r+1)$-intersecting. Since $r\leq d$ implies $2d-t-r+1\geq d+1-t$, by \eqref{ineq-2.2} and \eqref{ineq-walk} we see that $\ha(\overline{Q})$ is not pseudo $(2d-t-r+1)$-intersecting, by  Proposition \ref{fk2021} $\hb(R,Q)$ is pseudo $(2d-t-r+2)$-intersecting.  Thus by \eqref{ineq-walk}
\[
|\hb(R,Q)| \leq \binom{n-|X\cup T\cup Q|}{k-|S|-r-(2d-t-r+2)} =\binom{n-|X\cup T\cup Q|}{k-2d-2+t-|S|}.
\]
Since $t\leq d-1$, $|X\cup T\cup Q| \geq |S|+2d-t\geq |S|+d+1$ and
\[
\frac{n-d-1-|S|}{2}\geq k-d-1-|S|>k-2d-2+t-|S|,
\]
we infer that
\[
|\hb(R,Q)| \leq \binom{n-d-1-|S|}{k-2d-2+t-|S|}<\binom{n-d-1-|S|}{k-d-1-|S|}.
\]
Moreover,  $|\hb(R)|\leq \binom{n-d-1-|S|}{k-d-1-|S|}$ for $|R|=d+1$. Thus,
\begin{align*}
|\hb| &=\sum_{R\subset Q} |\hb(R,Q)| \\[5pt]
&\leq \sum_{R\subset Q, |R|\leq d} |\hb(R,Q)| +\sum_{R\subset Q, |R|= d+1} |\hb(R)| \\[5pt]
&< \sum_{0\leq i\leq d} \binom{2d-t}{i}\binom{n-d-1-|S|}{k-d-1-|S|} +\binom{2d-t}{d+1}\binom{n-d-1-|S|}{k-d-1-|S|}\\[5pt]
&\leq \binom{n-d-1-|S|}{k-d-1-|S|}\sum_{0\leq i\leq d+1} \binom{2d-1}{i}\\[5pt]
&\leq 2^{2d-1}\binom{n-d-1-|S|}{k-d-1-|S|}.
\end{align*}
\end{proof}
By \eqref{ineq-hfSEminusT}, we have
\begin{align*}
|\hf(\overline{T})| = \sum_{S\subset X\setminus T} |\hf(S,X\cup T)|
&<\sum_{0\leq j\leq |X\setminus T|}\binom{|X\setminus T|}{j}2^{2d-1}\binom{n-d-1-j}{k-d-1-j}\\[5pt]
&< \sum_{0\leq j\leq |X\setminus T|}\binom{2d^2-2d}{j}2^{2d-1}\binom{n-d-1-j}{k-d-1-j}.
\end{align*}
Note that $n\geq 4d(d-1)k$ implies
\[
\frac{\binom{2d^2-2d}{j+1}\binom{n-d-2-j}{k-d-2-j}}{\binom{2d^2-2d}{j}\binom{n-d-1-j}{k-d-1-j}} =\frac{(2d^2-2d-j)(k-d-1-j)}{(j+1)(n-d-1-j)}<\frac{(2d^2-2d)k}{n} \leq\frac{1}{2}.
\]
It follows that
\[
\sum_{0\leq j\leq |X\setminus T|}\binom{2d^2-2d}{j}\binom{n-d-1-j}{k-d-1-j} < \binom{n-d-1}{k-d-1}\sum_{0\leq i\leq \infty} 2^{-i} = 2\binom{n-d-1}{k-d-1}.
\]
Thus\begin{align*}
|\hf(\overline{T})| < 2^{2d} \binom{n-d-1}{k-d-1}< \frac{1}{d}|\hf|.
\end{align*}
We conclude
\[
\sum_{x\in T} |\hf(x)| \geq |\hf|-|\hf(\overline{T})| >\frac{d-1}{d}|\hf|.
\]
But $|T|=t\leq d-1$ and  we obtain that there exists some $x\in T$ with $|\hf(x)|>\frac{1}{d}|\hf|$, contradicting $\varrho(\hf)\leq \frac{1}{d}$.
\end{proof}
\section{Maximum degree results for cross-intersecting families}

In this section, we prove some maximum degree results for cross-intersecting families.

Let us first prove the $s=1$ version of Proposition \ref{prop-1.2} for two families.

\begin{fact}\label{fact-2}
Suppose that $\hf,\hg\subset 2^{[n]}$ are initial and cross-intersecting. Then $\hf(\bar{1})$ and $\hg(\bar{1})$ are cross 2-intersecting.
\end{fact}

\begin{proof}
Suppose for contradiction that for some $F\in \hf(\bar{1})$, $G\in \hg(\bar{1})$ and $2\leq j\leq n$, $F\cap G=\{j\}$ holds. Since $(F\setminus \{j\})\cup \{1\}=:F' \prec F$, $F'\in \hf$. However $F'\cap G=\emptyset$, a contradiction.
\end{proof}

\begin{prop}\label{prop-2.2}
Let $\hf,\hg\subset \binom{[n]}{k}$ be non-empty initial and cross-intersecting. Then
\begin{align}\label{ineq-maxrho}
\mbox{either }\max\left\{\varrho(\hf),\varrho(\hg)\right\}\geq\frac{k}{2k-2} \mbox{ or }\min\left\{\varrho(\hf),\varrho(\hg)\right\}\geq\frac{k}{2k-1}.
\end{align}
\end{prop}

\begin{proof}
For $\hh=\hf$ or $\hg$, since $\hh$ is initial, by \eqref{ineq-shifted} we have $\partial \hh(\bar{1})\subset \hh(1)$. By Fact \ref{fact-2}  $\hf(\bar{1}),\hg(\bar{1})$ are cross 2-intersecting. Then by Proposition \ref{fk2021} either both $\hf(\bar{1})$ and $\hg(\bar{1})$ are pseudo $2$-intersecting, or one of $\hf(\bar{1})$, $\hg(\bar{1})$ is pseudo $3$-intersecting.

If both $\hf(\bar{1})$ and $\hg(\bar{1})$ are pseudo $2$-intersecting, by \eqref{ineq-ishadow}   we infer
\[
\frac{|\partial \hf(\bar{1})|}{|\hf(\bar{1})|},\ \frac{|\partial \hg(\bar{1})|}{|\hg(\bar{1})|}\geq \frac{\binom{2k-2}{k-1}}{\binom{2k-2}{k}}=\frac{k}{k-1}.
\]
It follows that
\[
|\hf(1)|\geq |\partial \hf(\bar{1})|\geq |\hf(\bar{1})|\frac{k}{k-1},\ |\hg(1)|\geq |\partial \hg(\bar{1})|\geq |\hg(\bar{1})|\frac{k}{k-1}.
\]
Thus $\min\left\{\varrho(\hf),\varrho(\hg)\right\}\geq\frac{k}{2k-1}$.

If one of $\hf(\bar{1})$, $\hg(\bar{1})$ is pseudo $3$-intersecting, without loss of generality, assume that $\hf(\bar{1})$ is pseudo $3$-intersecting. Then
by \eqref{ineq-ishadow}
\[
\frac{|\partial \hf(\bar{1})|}{|\hf(\bar{1})|} \geq \frac{\binom{2k-3}{k-1}}{\binom{2k-3}{k}}=\frac{k}{k-2}.
\]
It follows that  $|\hf(1)|\geq |\hf(\bar{1})|\frac{k}{k-2}$ and hence  $\max\left\{\varrho(\hf),\varrho(\hg)\right\}\geq\frac{k}{2k-2}$.
\end{proof}

\begin{lem}\label{lem-main3}
Suppose that $n\geq 39k$, $\hf,\hg\subset \binom{[n]}{k}$ are cross-intersecting with $|\hf|,|\hg|\geq 2\binom{n-2}{k-2}+4\binom{n-3}{k-3}$. Then $\max\{\varrho(\hf),\varrho(\hg)\}>\frac{1}{2}$.
\end{lem}

Let us first prove  Theorems \ref{thm-main3} and \ref{thm-main4} by applying Lemma \ref{lem-main3}.

\begin{proof}[Proof of Theorem \ref{thm-main3}]
Let $\hf,\hg\subset \binom{[n]}{k}$ be cross-intersecting with $|\hf|,|\hg|\geq 2\binom{n-2}{k-2}+4\binom{n-3}{k-3}$. By Lemma \ref{lem-main3} we infer  $\max\{\varrho(\hf),\varrho(\hg)\}>\frac{1}{2}$. Without loss of generality, assume that
$\varrho(\hf)>\frac{1}{2}$ and $|\hf(1)|= \max\limits_{1\leq i\leq n} |\hf(i)|$. Then
\[
|\hf(1)|>\frac{1}{2} |\hf| \geq \binom{n-2}{k-2}+2\binom{n-3}{k-3}>\binom{n-2}{k-2}+\binom{n-4}{k-3}+\binom{n-5}{k-3}.
\]
Hence
\[
\hl([2,n],k-1,|\hf(1)|) \supset \left\{H\in \binom{[2,n]}{k-1}\colon 2\in H\right\}\bigcup \left\{K\in \binom{[3,n]}{k-1}\colon (3,4)\subset K\mbox{ or }(3,5)\subset K\right\}.
\]
Since $\hg(\bar{1})\subset \binom{[2,n]}{k}$ is cross-intersecting with $\hf(1)$, by Hilton's Lemma we infer that
\[
\hl([2,n],k,|\hg(\bar{1})|)\subset \left\{L\in \binom{[2,n]}{k}\colon (2,3)\subset L \mbox{ or }(2,4,5)\subset L\right\}.
\]
Therefore,
\[
|\hg(\bar{1})|\leq \binom{n-3}{k-2}+\binom{n-5}{k-3}<\binom{n-2}{k-2}.
\]
It follows that
\[
\varrho(\hg) \geq \frac{|\hg(1)|}{|\hg|} > \frac{|\hg|-\binom{n-2}{k-2}}{|\hg|}\geq \frac{\binom{n-2}{k-2}+4\binom{n-3}{k-3}}{2\binom{n-2}{k-2}+4\binom{n-3}{k-3}}
=\frac{1}{2}+\frac{\binom{n-3}{k-3}}{\binom{n-2}{k-2}+2\binom{n-3}{k-3}}
>\frac{1}{2}+\frac{k-2}{2(n-2)}.
\]
 Now using $\varrho(\hg)>\frac{1}{2}$ one can deduce by the same argument that $\varrho(\hf)> \frac{1}{2}+\frac{k-2}{2(n-2)}$ as well.
\end{proof}

\begin{proof}[Proof of Theorem \ref{thm-main4}]
Let $\hf,\hg\subset \binom{[n]}{k}$ be cross-intersecting with $|\hf|,|\hg|\geq |\hht(n,k)|$. Note that $n\geq 6k$ implies $\binom{n-2}{k-2}> 6\binom{n-3}{k-3}$. It follows that
\[
|\hht(n,k)|=3\binom{n-3}{k-2}+\binom{n-3}{k-3}=3\binom{n-2}{k-2}-2\binom{n-3}{k-3}> 2\binom{n-2}{k-2}+4\binom{n-3}{k-3}.
\]
By Lemma  \ref{lem-main3} we infer  $\max\{\varrho(\hf),\varrho(\hg)\}>\frac{1}{2}$. Without loss of generality, assume that $|\hf(1)|>\frac{1}{2}|\hf|$. Then
\[
|\hf(1)|> \frac{3}{2}\binom{n-2}{k-2}-\binom{n-3}{k-3}>\binom{n-2}{k-2}+\binom{n-4}{k-3}+\binom{n-5}{k-3}.
\]
Hence
\[
\hl([2,n],k-1,|\hf(1)|) \supset \left\{H\in \binom{[2,n]}{k-1}\colon 2\in H\right\}\bigcup \left\{K\in \binom{[3,n]}{k-1}\colon (3,4)\subset K\mbox{ or }(3,5)\subset K\right\}.
\]
Since $\hg(\bar{1})\subset \binom{[2,n]}{k}$ is cross-intersecting with $\hf(1)$, we infer that
\[
\hl([2,n],k,|\hg(\bar{1})|)\subset \left\{L\in \binom{[2,n]}{k}\colon (2,3)\subset L \mbox{ or }(2,4,5)\subset L\right\}.
\]
Therefore,
\[
|\hg(\bar{1})|\leq \binom{n-3}{k-2}+\binom{n-5}{k-3}<\binom{n-2}{k-2}.
\]
By $|\hg|\geq |\hht(n,k)|=\binom{n-2}{k-2}+2\binom{n-3}{k-2}$, we conclude that
\[
\varrho(\hg) \geq \frac{|\hg(1)|}{|\hg|} > \frac{|\hg|-\binom{n-2}{k-2}}{|\hg|}\geq \frac{2\binom{n-3}{k-2}}{3\binom{n-3}{k-2}+\binom{n-3}{k-3}}
=\frac{2}{3}\left(1-\frac{\binom{n-3}{k-3}}{\binom{n-2}{k-2}+2\binom{n-3}{k-2}}\right)
>\frac{2}{3}\left(1-\frac{k-2}{n-2}\right).
\]
Now using $\varrho(\hg)>\frac{1}{2}$ one can deduce by the same argument $\varrho(\hf)> \frac{2}{3}\left(1-\frac{k-2}{n-2}\right)$ as well.
\end{proof}

For initial cross-intersecting families we can obtain a stronger result. For the proof, we need a result of Tokushige concerning the product of cross $t$-intersecting families.

\begin{thm}[\cite{Tokushige}]
Let $k\geq t\geq 1$ and $\frac{k}{n}<1-\frac{1}{\sqrt[t]{2}}$. Suppose that $\ha,\hb \subset \binom{[n]}{k}$ are cross $t$-intersecting. Then
\begin{align}\label{ineq-prodtintersecting}
|\ha||\hb|\leq \binom{n-t}{k-t}^2.
\end{align}
\end{thm}

Let us mention that the proof in \cite{Tokushige} is based on eigenvalues. In \cite{FLST} for $t\geq 14$ the best possible bound \eqref{ineq-prodtintersecting} is established via combinatorial arguments for the full range, that is, $n\geq (t+1)(k-t+1)$.

\begin{prop}
Let $\hf,\hg\subset \binom{[n]}{k}$ be initial cross-intersecting families with $n\geq 3.5k$. If $\min\{|\hf|,|\hg|\}\geq |\hht(n,k)|$ then
\begin{align}\label{ineq-3.9}
\varrho(\hf)> \frac{2}{3}\mbox{ and } \varrho(\hg)> \frac{2}{3}.
\end{align}
\end{prop}
\begin{proof}
Suppose to the contrary  that $\varrho(\hf)\leq \frac{2}{3}$. Then
\[
|\hf(\bar{1})|\geq  \frac{1}{3}|\hf|\geq \frac{1}{3}|\hht(n,k)|> \binom{n-3}{k-2}.
\]
It follows that
\[
\left\{F\in \binom{[2,n]}{k}\colon \{2,3\}\subset F\right\}\subset \hl([2,n],k,|\hf(\bar{1})|).
\]
Since $\hf(\bar{1}),\hg(1)$ are cross-intersecting, by Hilton's Lemma we see that
\[
\hl([2,n],k-1,|\hg(1)|)\subset \left\{\bar{G}\in \binom{[2,n]}{k-1}\colon \{2,3\}\cap  \bar{G}\neq \emptyset\right\}.
\]
Therefore $|\hg(1)|\leq  \binom{n-2}{k-2}+\binom{n-3}{k-2}$. Since $|\hg|\geq  \hht(n,k)=2\binom{n-3}{k-2}+\binom{n-2}{k-2}$ it follows that
\[
|\hg(\bar{1})|\geq  \binom{n-3}{k-2}.
\]
But $\hf(\bar{1}),\hg(\bar{1})$ are cross 2-intersecting, which contradicts \eqref{ineq-prodtintersecting} for
\[
\frac{k}{n}\leq \frac{1}{3.5}<1-\frac{1}{\sqrt{2}}.
\]
\end{proof}

\begin{fact}
For $\hf\subset \binom{[n]}{k}$ and every $\{x,y\}\in \binom{[n]}{2}$,
\begin{align}\label{eq-hhxhhy}
|\hf(x)|+|\hf(y)| =|\hf(x,\bar{y})|+|\hf(\bar{x},y)|+2|\hf(x,y)|=|\hf(x,y)|+|\hf|-|\hf(\bar{x},\bar{y})|.  \square
\end{align}
\end{fact}

For the proofs of Lemma \ref{lem-main3} and Theorem \ref{thm-main2}, we need the following two lemmas.

\begin{lem}\label{lem3.7}
Suppose that $\hf,\hg\subset \binom{[n]}{k}$ are cross-intersecting and both of $\hf,\hg$ are initial on $[n-8]$.  Let
\begin{align*}
&\ha=\left\{A\in \binom{[n-8]}{k-1}\colon A\cup \{n-1\}\in \hf \mbox{ or }A\cup \{n\}\in \hf  \right\}.
\end{align*}
 If $|\ha(\bar{1},\bar{2})|>\binom{n-10}{k-3}$, then
\begin{align*}
|\hg(\overline{n-1},\bar{n})|<  |\hg(1,2)|+6\binom{n-3}{k-3}.
\end{align*}
\end{lem}

\begin{proof}
For any $S\subset [n-7,n-2]$, define $\hb_S=\hg(S,[n-7,n])\subset \binom{[n-8]}{k-|S|}$.
Since $\hf$, $\hg$ are both initial on $[n-8]$, $\ha$, $\hb_S$ are initial and cross-intersecting.
Note that  initiality implies $\hb_S(\bar{1},2)\subset \hb_S(1,\bar{2})$. Since $|\ha(\bar{1},\bar{2})|>\binom{n-10}{k-3}$, by \eqref{ineq-walk} we infer $\ha(\bar{1},\bar{2})$ is not pseudo 2-intersecting. As $\ha(\bar{1},\bar{2})$, $\hb_S(1,\bar{2})$ are cross 2-intersecting,  by   Proposition \ref{fk2021}, $\hb_S(1,\bar{2})$ is pseudo 3-intersecting. Then by \eqref{ineq-walk}
\[
|\hb_S(\bar{1},2)|\leq |\hb_S(1,\bar{2})|\leq \binom{n-10}{k-|S|-1-3}=\binom{n-10}{k-|S|-4}.
\]
By \eqref{ineq-walk}, $\ha(\bar{1},\bar{2})$ is not pseudo 3-intersecting. Since $\ha(\bar{1},\bar{2})$, $\hb_S(\bar{1},\bar{2})$ are cross 3-intersecting, by Proposition \ref{fk2021} we infer that $\hb_S(\bar{1},\bar{2})$ is  pseudo 4-intersecting. Then by \eqref{ineq-walk},
\[
|\hb_S(\bar{1},\bar{2})|\leq \binom{n-10}{k-|S|-4}.
\]

Then for $S=\emptyset$,
\begin{align*}
|\hb_{\emptyset}| =|\hb_{\emptyset}(1,2)|+|\hb_{\emptyset}(\bar{1},2)|+ |\hb_{\emptyset}(1,\bar{2})|+|\hb_{\emptyset}(\bar{1},\bar{2})|\leq |\hg(1,2)|+3\binom{n-10}{k-4}.
\end{align*}
For  $|S|\geq 1$, note that $|\hb_S(1,2)|\leq \binom{n-10}{k-|S|-2}$,
\begin{align*}
|\hb_{S}| =|\hb_{S}(1,2)|+|\hb_{S}(\bar{1},2)|+ |\hb_{S}(1,\bar{2})|+|\hb_{S}(\bar{1},\bar{2})|\leq \binom{n-10}{k-|S|-2}+3\binom{n-10}{k-|S|-4}.
\end{align*}
Thus
\begin{align*}
|\hg(\overline{n-1},\overline{n})| &=\sum_{S\subset [n-7,n-2]} |\hb_S|\\[5pt]
&\leq |\hg(1,2)|+3\binom{n-10}{k-4}+\sum_{1\leq i\leq 6} \binom{6}{i}\left(\binom{n-10}{k-i-2}+3\binom{n-10}{k-i-4}\right).
\end{align*}
Since
\[
\sum_{1\leq i\leq 6} \binom{6}{i}\binom{n-10}{k-i-2}=\binom{n-4}{k-2}-\binom{n-10}{k-2}
\]
and
\[
\sum_{1\leq i\leq 6} \binom{6}{i}\binom{n-10}{k-i-4}=\binom{n-4}{k-4}-\binom{n-10}{k-4},
\]
it follows that
\begin{align*}
|\hg(\overline{n-1},\overline{n})|&<|\hg(1,2)|+\binom{n-4}{k-2}
-\binom{n-10}{k-2}+3\binom{n-4}{k-4}\\[5pt]
&<|\hg(1,2)|+ 6\binom{n-4}{k-3}+3\binom{n-4}{k-4}\\[5pt]
&<|\hg(1,2)|+ 6\binom{n-3}{k-3}.
\end{align*}
\end{proof}

\begin{lem}\label{lem3.8}
Suppose that $\hf,\hg\subset \binom{[n]}{k}$ are cross-intersecting and both of $\hf,\hg$ are initial on $[n-8]$. Let $\ha_1=\hf(n,[n-7,n])$ and $\hb_2=\hg(n-1,[n-7,n])$. If $|\ha_1(\bar{1},\bar{2})|>\binom{n-10}{k-3}$, then
\begin{align*}
|\hb_2|<  \binom{n-3}{k-3}.
\end{align*}
\end{lem}

\begin{proof}
Note that initiality implies $\hb_2(\bar{1},2)\subset \hb_2(1,\bar{2})$. Since $|\ha_1(\bar{1},\bar{2})|>\binom{n-10}{k-3}$, by \eqref{ineq-walk} $\ha_1(\bar{1},\bar{2})$ is not pseudo 2-intersecting.  As $\ha_1(\bar{1},\bar{2})$, $\hb_2(1,\bar{2})$ are cross 2-intersecting,  by Proposition \ref{fk2021} we infer that $\hb_2(1,\bar{2})$ is  pseudo 3-intersecting. Then by \eqref{ineq-walk}
\[
|\hb_2(\bar{1},2)|\leq |\hb_2(1,\bar{2})|\leq \binom{n-10}{k-5}.
\]
As $\ha_1(\bar{1},\bar{2})$ is not pseudo 3-intersecting and $\ha_1(\bar{1},\bar{2})$, $\hb_2(\bar{1},\bar{2})$ are cross 3-intersecting,  Proposition \ref{fk2021} implies that $\hb_2(\bar{1},\bar{2})$ is  pseudo 4-intersecting. Then by \eqref{ineq-walk},
\[
|\hb_2(\bar{1},\bar{2})|\leq \binom{n-10}{k-5}.
\]
Clearly $|\hb_2(1,2)|\leq \binom{n-10}{k-3}$. Therefore,
\[
|\hb_2| =|\hb_2(1,2)|+|\hb_2(\bar{1},2)|+ |\hb_2(1,\bar{2})|+|\hb_2(\bar{1},\bar{2})|< \binom{n-3}{k-3}.
\]
\end{proof}

Now we are in a position to prove Lemma \ref{lem-main3}.

\begin{proof}[Proof of Lemma \ref{lem-main3}]
Arguing indirectly assume that $\hf,\hg\subset \binom{[n]}{k}$ are  cross-intersecting,
\begin{align}
|\hf|,|\hg|\geq 2\binom{n-2}{k-2}+4\binom{n-3}{k-3}\mbox{ and } \varrho(\hf),\varrho(\hg)\leq \frac{1}{2}.
\end{align}
Without loss of generality, suppose that
$\hf,\hg$ are shifted ad extremis with respect to $\{\varrho(\hf)\leq \frac{1}{2},\varrho(\hg)\leq \frac{1}{2}\}$ and let $\mathds{H}_1$, $\mathds{H}_2$ be the graphs formed by the shift-resistant pairs for $\hf$ and $\hg$, respectively.

Let $\hh=\hf$ or $\hg$.  For every $\{x,y\}\in \binom{[n]}{2}$, we may assume that $|\hh(x,y)|\leq |\hh(\bar{x},\bar{y})|$. Otherwise by \eqref{eq-hhxhhy} $|\hh(x)|+|\hh(y)|> |\hh|$ and  we are done.

\begin{claim}
We may assume that  for all $P\in \binom{[n]}{2}$
\begin{align}\label{ineq-hfp2}
|\hf(P)|,|\hg(P)|\leq \binom{n-3}{k-3}+\binom{n-4}{k-3}+\binom{n-5}{k-3}+\binom{n-7}{k-4}<3\binom{n-3}{k-3}.
\end{align}
\end{claim}
\begin{proof}
By symmetry, suppose that $|\hf(x,y)|>\binom{n-3}{k-3}+\binom{n-4}{k-3}+\binom{n-5}{k-3}+\binom{n-7}{k-4}$ for some $\{x,y\}\in \binom{[n]}{2}$.
For convenience assume that $\{x,y\}=\{n-1,n\}$. Then by \eqref{ineq-hgxy2} we obtain
\[
 |\hg(\overline{n-1},\overline{n})|\leq \binom{n-6}{k-4}+\binom{n-7}{k-4}<2\binom{n-4}{k-4}.
\]
Note that $\varrho(\hg)\leq \frac{1}{2}$ implies that
\[
|\hg(\overline{n-1})|\geq \frac{1}{2}|\hg| \geq \binom{n-2}{k-2}+2\binom{n-3}{k-3}=\binom{n-3}{k-2}+3\binom{n-3}{k-3}.
\]
It follows that
\begin{align*}
|\hg(\overline{n-1},n)|&= |\hg(\overline{n-1})|-  |\hg(\overline{n-1},\overline{n})|\\[5pt]
&\geq \binom{n-3}{k-2}+3\binom{n-3}{k-3}-2\binom{n-4}{k-4}\\[5pt]
&>\binom{n-3}{k-2}+\binom{n-5}{k-3}+\binom{n-6}{k-3}.
\end{align*}
Similarly, we have
\[
|\hg(n-1,\overline{n})|\geq \binom{n-3}{k-2}+\binom{n-5}{k-3}+\binom{n-6}{k-3}.
\]
Clearly,
\[
\left\{\bar{G}\in \binom{[n-2]}{k-1}\colon 1\in \bar{G} \mbox{ or }\{2,3\}\subset \bar{G} \mbox{ or }\{2,4\}\subset \bar{G} \right\}\subset \hl(n-2,k-1,|\hg(\overline{n-1},n)|).
\]
Since $\hf(n-1,\overline{n}), \hg(\overline{n-1}, n)$ are cross-intersecting, by Hilton's Lemma
\[
 \hl(n-2,k-1,|\hf(n-1,\overline{n})|) \subset \left\{\bar{F}\in \binom{[n-2]}{k-1}\colon (1,2)\subset  \bar{F} \mbox{ or }(1,3,4)\subset \bar{F} \right\}.
\]
It follows that $|\hf(n-1,\overline{n})|\leq \binom{n-4}{k-3}+\binom{n-6}{k-4}$. Similarly $|\hf(\overline{n-1},n)|\leq \binom{n-4}{k-3}+\binom{n-6}{k-4}$. Moreover, $\hf(\overline{n-1},\overline{n}), \hg(n-1,\overline{n})$ are cross-intersecting. By Hilton's Lemma
\[
 \hl(n-2,k,|\hf(\overline{n-1},\overline{n})|) \subset \left\{\tilde{F}\in \binom{[n-2]}{k}\colon (1,2)\subset  \tilde{F}\mbox{ or }(1,3,4)\subset \tilde{F} \right\}.
\]
Hence $|\hf(\overline{n-1},\overline{n})|\leq \binom{n-4}{k-2}+\binom{n-6}{k-3}$. By $|\hf(n-1,n)|\leq |\hf(\overline{n-1},\overline{n})|$, we conclude that
\begin{align*}
|\hf| &= |\hf(n-1,n)|+|\hf(\overline{n-1},\overline{n})|+|\hf(n-1,\overline{n})|+|\hf(\overline{n-1},n)|\\[5pt]
&\leq 2\left(\binom{n-4}{k-2}+\binom{n-6}{k-3}\right)+2\left(\binom{n-4}{k-3}+\binom{n-6}{k-4}\right)\\[5pt]
&=2\left(\binom{n-3}{k-2}+\binom{n-5}{k-3}\right)\\[5pt]
&<2\binom{n-2}{k-2}+4\binom{n-3}{k-3},
\end{align*}
contradiction.
\end{proof}

\begin{claim}\label{claim-8}
$\mathds{H}_i$ has matching number at most 2 for $i=1,2$.
\end{claim}
\begin{proof}
Suppose for contradiction that $(a_1,b_1), (a_2,b_2),(a_3,b_3)$ are pairwise disjoint and say
\[
\varrho(S_{a_rb_r}(\hf))>\frac{1}{2}\mbox{ for } r=1,2,3.
\]
Recall the definition
\[
\hf_{\{x,y\}} =\{F\in \hf\colon F\cap \{x,y\}\neq \emptyset\}.
\]
Set $\hj_r=\hf_{\{a_r,b_r\}}$, $r=1,2,3$. Then  for $r=1,2,3$,
\begin{align*}
|\hj_r| \geq & |\hf(a_r,b_r)|+|\hf(a_r,\overline{b_r})\cup \hf(\overline{a_r},b_r)|>  \frac{1}{2}|\hf|.
\end{align*}
Note that  \eqref{ineq-hfp2} and $n\geq 36k$ imply for $1\leq r<r'\leq 3$ that
\[
|\hj_r\cap \hj_{r'}|< 4\cdot 3\binom{n-3}{k-3}< \frac{1}{3}\binom{n-2}{k-2} < \frac{1}{6}|\hf|.
\]
Therefore,
\begin{align*}
|\hj_1\cup\hj_2\cup \hj_3| &> \frac{3}{2}|\hf|-|\hj_1\cap \hj_2|-|\hj_1\cap \hj_3|-|\hj_2\cap \hj_3|> \frac{3}{2}|\hf|-\frac{1}{2}|\hf|=|\hf|,
\end{align*}
a contradiction.
\end{proof}

If both $\hf$ and $\hg$ are initial, then by \eqref{ineq-maxrho} we are done. Thus we may assume that $\mathds{H}_1\cup \mathds{H}_2\neq\emptyset$. Without loss of generality, assume that $(n-1,n)\in \mathds{H}_1$. By Claim \ref{claim-8}, we may further assume that both $\hf$ and $\hg$ are initial on $[n-8]$. Let
\begin{align*}
&\ha=\left\{A\in \binom{[n-8]}{k-1}\colon A\cup \{n-1\}\in \hf \mbox{ or }A\cup \{n\}\in \hf  \right\}.
\end{align*}
Note that $(n-1,n)\in \mathds{H}_1$ implies
\begin{align}\label{ineq-3.2}
|\hf_{\{n-1,n\}}|=|\hf(n-1,n)|+|\hf(\overline{n-1},n)\cup \hf(n-1,\overline{n})|> \frac{1}{2}|\hf|.
\end{align}
For $1\leq i\leq 3$,  applying Lemma \ref{lem-3.5} with $M=\binom{n-3}{k-3}+\binom{n-4}{k-3}+\binom{n-5}{k-3}+\binom{n-7}{k-4}>2\binom{n-5}{k-3}$, from \eqref{ineq-FQR-general} we infer
\begin{align*}
|\hf_{\{n-1,n\}}\cap \hf_{\{n-2i-1,n-2i\}}|
\leq 3M+\binom{n-7}{k-5}+\binom{n-8}{k-5}< 9\binom{n-3}{k-3}.
\end{align*}
 By \eqref{ineq-hfp2} and \eqref{ineq-3.2}, it follows that
\begin{align*}
|\ha| &> \frac{1}{2}|\hf| -|\hf(n-1,n)| -\sum_{1\leq i\leq 3}|\hf_{\{n-1,n\}}\cap \hf_{\{n-2i-1,n-2i\}}|\\[5pt]
&> \frac{1}{2}|\hf| - 30\binom{n-3}{k-3}.
\end{align*}
By Lemma \ref{lem-3.5}, we also have $|\hf_{\{n-1,n\}}\cap \hf_{\{1,2\}}|< 9\binom{n-3}{k-3}$. Then for $n\geq 39k$,
\begin{align}\label{ha1bar2bar2}
|\ha(\bar{1},\bar{2})|&\geq |\ha| - |\hf_{\{n-1,n\}}\cap \hf_{\{1,2\}}|\nonumber\\[5pt]
&> \frac{1}{2}|\hf| - 39\binom{n-3}{k-3}\nonumber\\[5pt]
&\geq  \binom{n-2}{k-2}+2\binom{n-3}{k-3} -39\binom{n-3}{k-3}\nonumber\\[5pt]
&>2\binom{n-10}{k-3}.
\end{align}
By Lemma \ref{lem3.7}, we obtain that
\begin{align}\label{ineq-3.11}
|\hg(\overline{n-1},\overline{n})|< |\hg(1,2)|+6\binom{n-3}{k-3}\overset{\eqref{ineq-hfp2}}{<}9\binom{n-3}{k-3}.
\end{align}
Since $|\hg(\overline{n})|\geq \frac{1}{2}|\hg|$, by \eqref{ineq-3.11} it follows that
\[
|\hg(n-1,\overline{n})| =|\hg(\overline{n})|- |\hg(\overline{n-1},\overline{n})|> \frac{1}{2}|\hg|-9\binom{n-3}{k-3}.
\]
Then for $n\geq 26k$
\begin{align}\label{ineq -hb_2}
|\hg(n-1,[n-7,n])| &\geq |\hg(n-1,\overline{n})| - \sum_{i\in[n-7,n-2]}|\hg(n-1,i)|\nonumber\\[5pt]
&> \frac{1}{2}|\hg|-27\binom{n-3}{k-3}\nonumber\\[5pt]
&\geq \binom{n-2}{k-2}+2\binom{n-3}{k-3} -27\binom{n-3}{k-3}\nonumber\\[5pt]
&> \binom{n-3}{k-3}.
\end{align}

Let
\[
\ha_i= \{A\in \ha\colon A\cup \{n+1-i\}\in \hf\}, i=1,2.
\]
Clearly $|\ha(\bar{1},\bar{2})|=|\ha_1(\bar{1},\bar{2})|+|\ha_2(\bar{1},\bar{2})|$. Without loss of generality, assume that $|\ha_1(\bar{1},\bar{2})|\geq |\ha_2(\bar{1},\bar{2})|$. Then by \eqref{ha1bar2bar2} we infer
\begin{align*}
|\ha_1(\bar{1},\bar{2})|>\binom{n-10}{k-3}.
\end{align*}
Now by  Lemma \ref{lem3.8} we obtain that
\[
|\hg(n-1,[n-7,n])|<\binom{n-3}{k-3},
\]
contradicting \eqref{ineq -hb_2}.
\end{proof}

By following a similar approach as in the proof of Lemma \ref{lem-main3}, we prove Theorem \ref{thm-main2}.

\begin{proof}[Proof of Theorem \ref{thm-main2}]
Arguing indirectly assume that $\hf,\hg\subset \binom{[n]}{k}$ are  cross-intersecting,
\begin{align}\label{ineq-3.1}
|\hf|,|\hg|\geq \frac{1}{\varepsilon}\binom{n-3}{k-3} \mbox{ and } \varrho(\hf),\varrho(\hg)\leq \frac{1}{2}-\varepsilon.
\end{align}
Without loss of generality suppose that
$\hf,\hg$ are shifted ad extremis with respect to $\{\varrho(\hf)\leq \frac{1}{2}-\varepsilon,\varrho(\hg)\leq \frac{1}{2}-\varepsilon\}$ and let $\mathds{H}_1, \mathds{H}_2$ be the graphs formed by the shift-resistant pairs in $\hf$ and $\hg$, respectively.

\begin{claim}
We may assume that  for all $P\in \binom{[n]}{2}$
\begin{align}\label{ineq-hfp}
|\hf(P)|,|\hg(P)|\leq \binom{n-3}{k-3}+\binom{n-4}{k-3}+\binom{n-6}{k-4}<2\binom{n-3}{k-3}.
\end{align}
\end{claim}
\begin{proof}
By symmetry, assume that $|\hf(x,y)|>\binom{n-3}{k-3}+\binom{n-4}{k-3}+\binom{n-6}{k-4}$ for some $\{x,y\}\in \binom{[n]}{2}$. Then by \eqref{ineq-hgxy} we obtain
\[
|\hg(\bar{x},\bar{y})|\leq \binom{n-5}{k-3}+\binom{n-6}{k-3}<2\binom{n-3}{k-3}\leq 2\varepsilon |\hg|.
\]
By \eqref{eq-hhxhhy}, it follows that
\[
|\hg(x)|+|\hg(y)| \geq |\hg|-|\hg(\bar{x},\bar{y})|> (1-2\varepsilon)|\hg|.
\]
Therefore,
\[
\mbox{either }|\hg(x)|>\left(\frac{1}{2}-\varepsilon\right)|\hg| \mbox{ or } |\hg(y)|>\left(\frac{1}{2}-\varepsilon\right)|\hg|,
\]
contradicting our assumption \eqref{ineq-3.1}.
\end{proof}

\begin{claim}\label{claim-6}
$\mathds{H}_i$ has matching number at most 2 for $i=1,2$.
\end{claim}
\begin{proof}
Suppose for contradiction that $(a_1,b_1), (a_2,b_2),(a_3,b_3)$ are pairwise disjoint and say
\[
\varrho(S_{a_rb_r}(\hf))>\frac{1}{2} -\varepsilon\mbox{ for } r=1,2,3.
\]
Define $\hj_r=\hf_{\{a_r,b_r\}}$, $r=1,2,3$. Then  for $r=1,2,3$,
\begin{align*}
|\hj_r| \geq & |\hf(a_r,b_r)|+|\hf(a_r,\overline{b_r})\cup \hf(\overline{a_r},b_r)|>  \left(\frac{1}{2}-\varepsilon\right)|\hf|.
\end{align*}
Note that  \eqref{ineq-hfp} implies for $1\leq r<r'\leq 3$ that
\[
|\hj_r\cap \hj_{r'}|< 4\cdot 2\binom{n-3}{k-3} \leq 8\varepsilon|\hf|.
\]
By $\varepsilon \leq \frac{1}{58}<\frac{1}{54}$, we have
\begin{align*}
|\hj_1\cup\hj_2\cup \hj_3| &> \frac{3}{2}|\hf|-3\varepsilon |\hf|-|\hj_1\cap \hj_2|-|\hj_1\cap \hj_3|-|\hj_2\cap \hj_3|\\[5pt]
&> \frac{3}{2}|\hf|-3\varepsilon|\hf|-24\varepsilon|\hf|\\[5pt]
&= \frac{3}{2}|\hf|-27\varepsilon|\hf|\\[5pt]
&> |\hf|,
\end{align*}
a contradiction.
\end{proof}

If both $\hf$ and $\hg$ are initial, then by Proposition \ref{prop-2.2} we are done. Thus we may assume that $\mathds{H}_1\cup\mathds{H}_2\neq\emptyset$. Without loss of generality, assume that $(n-1,n)\in \mathds{H}_1$. By Claim \ref{claim-6}, we may further assume that both $\hf$ and $\hg$ are initial on $[n-8]$. Let
\begin{align*}
&\ha=\left\{A\in \binom{[n-8]}{k-1}\colon A\cup \{n-1\}\in \hf \mbox{ or }A\cup \{n\}\in \hf  \right\}.
\end{align*}
Note that $(n-1,n)\in \hh_1$ implies
\begin{align}\label{ineq-4.11}
|\hf_{\{n-1,n\}}|=|\hf(n-1,n)|+|\hf(\overline{n-1},n)\cup \hf(n-1,\overline{n})|> \left(\frac{1}{2}-\varepsilon\right)|\hf|.
\end{align}
For $1\leq i\leq 3$,  applying Lemma \ref{lem-3.5} with $M=\binom{n-3}{k-3}+\binom{n-4}{k-3}+\binom{n-6}{k-4}>2\binom{n-5}{k-3}$, from \eqref{ineq-FQR-general} we infer
\begin{align*}
|\hf_{\{n-1,n\}}\cap \hf_{\{n-2i-1,n-2i\}}|
&\leq 3M+\binom{n-7}{k-5}+\binom{n-8}{k-5}< 6\binom{n-3}{k-3}.
\end{align*}
 By \eqref{ineq-4.11} and \eqref{ineq-hfp}, it follows that
\begin{align*}
|\ha| &> \left(\frac{1}{2}-\varepsilon\right)|\hf| -|\hf(n-1,n)| -\sum_{1\leq i\leq 3}|\hf_{\{n-1,n\}}\cap \hf_{\{n-2i-1,n-2i\}}|\\[5pt]
&> \left(\frac{1}{2}-\varepsilon\right)|\hf| - 20\binom{n-3}{k-3}.
\end{align*}
By Lemma \ref{lem-3.5} we also have $|\hf_{\{n-1,n\}}\cap \hf_{\{1,2\}}|< 6\binom{n-3}{k-3}$. Since $\varepsilon \leq \frac{1}{58}$,
\begin{align}\label{ha1bar2bar}
|\ha(\bar{1},\bar{2})|&\geq |\ha| - |\hf_{\{n-1,n\}}\cap \hf_{\{1,2\}}|\nonumber\\[5pt]
&> \left(\frac{1}{2}-\varepsilon\right)|\hf| - 26\binom{n-3}{k-3}\nonumber\\[5pt]
&\geq \frac{1}{2\varepsilon}\binom{n-3}{k-3}-\binom{n-3}{k-3} - 26\binom{n-3}{k-3}\nonumber\\[5pt]
&\geq 2\binom{n-10}{k-3}.
\end{align}
By Lemma \ref{lem3.7}, we obtain that
\begin{align}\label{ineq-3.17}
|\hg(\overline{n-1},\overline{n})|< |\hg(1,2)|+6\binom{n-3}{k-3}\overset{\eqref{ineq-hfp}}{<}8\binom{n-3}{k-3}.
\end{align}
Since $|\hg(\overline{n})|> \frac{1}{2}|\hg|$, by \eqref{ineq-3.17} it follows that
\[
|\hg(n-1,\overline{n})| =|\hg(\overline{n})|- |\hg(\overline{n-1},\overline{n})|> \frac{1}{2}|\hg|-8\binom{n-3}{k-3}.
\]
Then for $\varepsilon \leq \frac{1}{58}<\frac{1}{42}$ we have
\begin{align}\label{ineq -hb_22}
|\hg(n-1,[n-7,n])| &\geq |\hg(n-1,\overline{n})| - \sum_{i\in[n-7,n-2]}|\hg(n-1,i)|\nonumber\\[5pt]
&> \frac{1}{2}|\hg|-20\binom{n-3}{k-3}\nonumber\\[5pt]
&\geq \frac{1}{2\varepsilon}\binom{n-3}{k-3}-20\binom{n-3}{k-3}\nonumber\\[5pt]
&> \binom{n-3}{k-3}.
\end{align}

Let
\[
\ha_i= \{A\in \ha\colon A\cup \{n+1-i\}\in \hf\}, i=1,2.
\]
Clearly $|\ha(\bar{1},\bar{2})|=|\ha_1(\bar{1},\bar{2})|+|\ha_2(\bar{1},\bar{2})|$. Without loss of generality, assume that $|\ha_1(\bar{1},\bar{2})|\geq |\ha_2(\bar{1},\bar{2})|$. Then by \eqref{ha1bar2bar} we infer
\begin{align*}
|\ha_1(\bar{1},\bar{2})|>\binom{n-10}{k-3}.
\end{align*}
Now Lemma \ref{lem3.8} implies
\[
|\hg(n-1,[n-7,n])|< \binom{n-3}{k-3},
\]
contradicting \eqref{ineq -hb_22}.
\end{proof}

In the rest of this section, we give some examples and results for maximum degree problems concerning initial cross-intersecting families.

\begin{example}\label{examp-2}
Let $\hg=\binom{[k+1]}{k}$ and $\hf=\{F\in \binom{[n]}{k}\colon |F\cap [k+1]|\geq 2\}$. Clearly, $\hg,\hf$ are initial and cross-intersecting. Note that $\varrho(\hg)=\frac{k}{k+1}$, $\varrho(\hf)=\frac{2}{k+1}+o(1)$. In particular,
\[
\varrho(\hg)+\varrho(\hf) =\frac{k+2}{k+1}+o(1)>1.
\]
Moreover,
\[
|\hg|+|\hf| =\binom{k+1}{k}+\sum_{2\leq i\leq k} \binom{k+1}{i}\binom{n-k-1}{k-i}.
\]
\end{example}

Set
\begin{align*}
g(n,k) &=\binom{k+1}{k}+\sum_{2\leq i\leq k} \binom{k+1}{i}\binom{n-k-1}{k-i}=\left(\binom{k+1}{2}+o(1)\right)\binom{n-2}{k-2}.
\end{align*}
Let us mention a recent result.

\begin{thm}[\cite{F2022}]
If $\hf,\hg \subset \binom{[n]}{k}$ are cross-intersecting, initial, $n\geq 2k$ then
\[
|\hf|+|\hg|\leq g(n,k),
\]
 and for $n> 2k>6$ in case of equality $(\hf, \hg)$ are as in  Example \ref{examp-2}.
\end{thm}

\begin{example}\label{examp-3}
Fix $k,\ell\geq 1$, $n>k+\ell$ and  $q<k+\ell$. Set
\[
\he(n,k,q,a)=\left\{E\subset \binom{[n]}{k}\colon |E\cap [q]|\geq a\right\}.
\]
Suppose that $a+b=q+1$. Then $\he(n,k,q,a)$ and $\he(n,k,q,b)$ are initial, cross-intersecting and
\[
\varrho(\he(n,k,q,a))=\frac{a}{q}+o(1), \ \varrho(\he(n,k,q,b))=\frac{b}{q}+o(1),
\]
i.e., $\varrho(\he(n,k,q,a))+\varrho(\he(n,k,q,b))\sim 1+\frac{1}{q}$.
\end{example}

Let us mention an old result showing that the example is in some sense best possible.

\begin{prop}[\cite{F87}]
Suppose that $\hf,\hg\subset \binom{[n]}{k}$ are initial and cross-intersecting. Then for any $F\in \hf, G\in\hg$, there exists $q=q(F,G)$ such that
\begin{align}
|F\cap [q]| +|G\cap [q]|\geq q+1.
\end{align}
\end{prop}

For cross-intersecting families of relatively large size one can bound $\varrho$ for each of them.

\begin{prop}
Suppose that $\hf,\hg\subset \binom{[n]}{k}$ are initial, cross-intersecting and
\[
\min\{|\hf|,|\hg|\}>\binom{n}{k-3}.
\]
Then
\[
\varrho(\hf)\geq \frac{1}{2} \mbox{ and } \varrho(\hg)\geq \frac{1}{2}.
\]
\end{prop}

\begin{proof}
In view of \eqref{ineq-walk}, $\min\{|\hf|,|\hg|\}>\binom{n}{k-3}$ implies
\[
(1,2,4,6,\ldots,2k-2)\in \hf\cap \hg.
\]
It follows that $(3,5,\ldots,2k+1)\notin \hf\cup \hg$, i.e.,
$\hf(\bar{1})$ and $\hg(\bar{1})$ are pseudo intersecting. Then by \eqref{ineq-shifted} and \eqref{ineq-ishadow}  we infer
\[
|\hf(1)|\geq |\partial \hf(\bar{1})|\geq |\hf(\bar{1})|\mbox{ and } |\hg(1)|\geq |\partial \hg(\bar{1})|\geq |\hg(\bar{1})|,
 \]
 implying the statement.
\end{proof}

To prove \eqref{ineq-conj} below we need a simple analytic inequality.

\begin{fact}\label{fact-3.13}
Let $0<a\leq A$, $0<b\leq B$ then
\[
\frac{a+b}{A+B} \geq \min\left\{\frac{a}{A},\frac{b}{B}\right\}.
\]
\end{fact}
\begin{proof}
Without loss of generality $\frac{a}{A}\leq \frac{b}{B}$. Set $b'=\frac{aB}{A}$ and note $b'\leq b$. Set $c=\frac{a}{A}=\frac{b'}{B}$. Then
\[
\frac{a+b}{A+B} \geq \frac{a+b'}{A+B}=c=\frac{a}{A}.
\]
\end{proof}

\begin{prop}\label{lem-1}
Suppose that $\hf\subset \binom{[n]}{k}$, $\hg\subset \binom{[n]}{\ell}$  and  $\hf$, $\hg$ are non-empty, initial and cross-intersecting then
\begin{align}\label{ineq-conj}
\varrho(\hf)+\varrho(\hg) \geq 1.
\end{align}
\end{prop}
\begin{proof}
Note that we do not assume $n\geq k+\ell$. Indeed, for $n\leq k+\ell$ by averaging $\varrho(\hf)\geq \frac{k}{n}$, $\varrho(\hg)\geq \frac{\ell}{n}$ and these imply \eqref{ineq-conj}. From now on we may assume $n>k+\ell$ and use induction on $n$.

The cases $k=1,\ell=1$ are trivial.  Consider the two vs two families $\hf(n)$, $\hf(\bar{n})$ and  $\hg(n)$, $\hg(\bar{n})$ which are cross-intersecting and initial. The cross-intersecting property holds trivially for three of the pairs without the assumption of initiality. As to the pair $\hf(n),\hg(n)$ initiality implies the cross-intersecting property (cf. \cite{F87}). Set $|\hf(n)|=R_1$, $|\hf(\bar{n})|=R_0$, $|\hg(n)|=Q_1$, $|\hg(\bar{n})|=Q_0$. Set also $|\hf(1,n)|=r_1$, $|\hf(1,\bar{n})|=r_0$,  $|\hg(1,n)|=q_1$, $|\hg(1,\bar{n})|=q_0$. Note that $\hf\neq \emptyset \neq \hg$ implies $R_0,Q_0\geq 1$.
 If $Q_1=0=R_1$, then $\hf,\hg\subset 2^{[n-1]}$. By the induction hypothesis,
 \[
 \varrho(\hf)+\varrho(\hg)= \varrho(\hf(\bar{n}))+\varrho(\hg(\bar{n}))\geq 1.
 \]

If $Q_1\neq 0$, $R_1\neq 0$, by the induction hypothesis we have
\begin{align}\label{ineq-IH}
\frac{r_i}{R_i}+\frac{q_j}{Q_j} \geq 1 \mbox{ for $i=0,1$ and $j=0,1$.}
\end{align}
By Fact \ref{fact-3.13},
\begin{align*}
\varrho(\hf)& =\frac{r_0+r_1}{R_0+R_1}\geq \min\left\{\frac{r_i}{R_i}\colon i=0,1\right\},\\[5pt]
\varrho(\hg)& =\frac{q_0+q_1}{Q_0+Q_1}\geq \min\left\{\frac{q_j}{Q_j}\colon j=0,1\right\}.
\end{align*}
Using \eqref{ineq-IH},  $\varrho(\hf)+\varrho(\hg)\geq 1$ follows.  If one of $Q_1,R_1$ equals 0, say $Q_1=0,R_1\neq 0$, then
 \[
 \varrho(\hg) =\frac{q_0}{Q_0}.
 \]
 Using \eqref{ineq-IH} with $j=0$, $\varrho(\hf)+\varrho(\hg)\geq 1$ follows.
\end{proof}

The finite projective plane is the standard example showing that $\varrho(\hf)$ can be quite small even for intersecting families.

\begin{example}
Let $q\geq 2$ be a prime power and consider
\[
\hf=\hg=\{\mbox{the lines of the projective plane of order } q\}.
\]
Note that this is an intersecting $(q+1)$-graph on $q^2+q+1$ vertices and regular of degree $q+1$.
Thus
\[
\varrho(\hf)=\varrho(\hg)=\frac{q+1}{q^2+q+1} \mbox{ and } \varrho(\hf)+\varrho(\hg)=\frac{2(q+1)}{q^2+q+1}<1.
\]
\end{example}

 For $\hf\subset \binom{[n]}{k}$, $\hg\subset \binom{[n]}{\ell}$  and  $\hf$, $\hg$  non-empty and cross-intersecting,  one can only get
 \[
 \varrho(\hf)+\varrho(\hg) \geq \frac{1}{\ell} +\frac{1}{k}
 \]
 and this is best possible as shown by the next example.

\begin{example}
For $k,\ell\geq 2$ define
\begin{align*}
\hp &= \left\{P_1,\ldots,P_\ell\colon P_1,\ldots,P_\ell \mbox{ pairwise disjoint $k$-sets}\right\},\\[5pt]
\hr &=\hht^{(\ell)}(\hp) =\left\{R\colon |R|=\ell, |R\cap P_i|=1, 1\leq i\leq \ell\right\}.
\end{align*}
Then $\hp$ and $\hr$ are cross-intersecting, regular and  $\varrho(\hp)=\frac{1}{\ell}$, $\varrho(\hr)=\frac{1}{k}$.

On the other hand, if $\emptyset \neq \hp \subset \binom{[n]}{k}$, $\emptyset \neq \hr \subset \binom{[n]}{\ell}$ and $\hp,\hr$ are cross-intersecting then one has always $\varrho(\hp)\geq \frac{1}{\ell}$, $\varrho(\hr)\geq \frac{1}{k}$.
\end{example}

\section{Maximum degree results for $t$-intersecting families}

In this section we consider $t$-intersecting families.

Let us recall the $t$-covering number $\tau_t(\hf)$:
\[
\tau_t(\hf) =\min\left\{|T|\colon |T\cap F|\geq t \mbox{ for all } F\in \hf \right\}.
\]
It should be clear that $\tau_t(\hf)=t$ if and only if $\hf$ is a $t$-star. Proposition \ref{prop-1.3} yields
\begin{align}\label{ineq-4.-1}
\varrho(\hf)\geq \frac{t}{\tau_t(\hf)}
\end{align}
for any $t$-intersecting family $\hf\subset \binom{[n]}{k}$.

In the case $\tau_t(\hf)=t+1$ one can improve on \eqref{ineq-4.-1}.

\begin{prop}
Suppose that $\hf\subset \binom{[n]}{k}$ is $t$-intersecting, $n\geq 2k$, $\tau_t(\hf)\leq t+1$ and $\hf$ is saturated. Then $\varrho(\hf)>\frac{t+1}{t+2}$.
\end{prop}

\begin{proof}
Without loss of generality let $[t+1]$ be a $t$-transversal of $\hf$, i.e., $|F\cap [t+1]|\geq t$ for all $F\in \hf$. Define
\[
\hf_i =\{F\setminus [t+1]\colon F\in \hf,F\cap [t+1]=[t+1]\setminus \{i\}\}
\]
and $\hf_0=\hf([t+1])$. By saturatedness $\hf_0=\binom{[t+2,n]}{k-t-1}$. Obviously, $\hf_i,\hf_j$ are cross-intersecting for $1\leq i<j\leq t+1$. By Hilton's Lemma, $\min\{|\hf_i|,|\hf_j|\}\leq \binom{n-t-2}{k-t-1}$. Assume by symmetry $|\hf_1|\leq |\hf_2|\leq \ldots\leq |\hf_{t+1}|$. Then
\[
|\hf_1|\leq \binom{n-t-2}{k-t-1} <\binom{n-t-1}{k-t-1}=|\hf_0|.
\]
Note that
\[
|\hf(1)| =|\hf_2|+\ldots+|\hf_{t+1}|+|\hf_0|>(t+1)|\hf_1|,\ |\hf(\bar{1})|=|\hf_1|.
\]
Thus
\[
\varrho(\hf)\geq \frac{|\hf(1)|}{|\hf(1)|+|\hf(\bar{1})|}>\frac{t+1}{t+2}.
\]
\end{proof}

{\bf\noindent Remark.} Considering all $k$-subsets of $[2k-t]$ shows that without some conditions on $|\hf|$ one cannot hope to prove better than $\varrho(\hf)\geq \frac{k}{2k-t}$.

Recently, the first author and Katona \cite{FK2021} proved the following improved version of  \eqref{ineq-ishadow} provided that $|\hf|\geq \left(1+\frac{t-1}{k+t}\right)\binom{2k-t}{k}$.

\begin{thm}[\cite{FK2021}]
Suppose that $\hf\subset \binom{[n]}{k}$ is $t$-intersecting, $1\leq \ell<t<k$, $|\hf|\geq \binom{2k-t}{k}\left(1+\frac{t+\ell}{k+t+1-\ell}\right)$. Then
\begin{align}\label{ineq-improvedKatona}
|\partial^\ell \hf|/|\hf| \geq \binom{2(k-1)-t}{k-1-\ell}/\binom{2(k-1)-t}{k-1}.
\end{align}
\end{thm}

In the case of cross $t$-intersecting families, $t\geq 2$, we cannot apply Hilton's Lemma. To circumvent this difficulty we prove a similar albeit somewhat weaker inequality.

\begin{prop}\label{thm-4.3}
Let $n,k,\ell,t,s$ be integers, $s> t\geq 2$, $k,\ell>s$. Suppose that $\hf\subset \binom{[n]}{k}$ and $\hg\subset \binom{[n]}{\ell}$ are cross $t$-intersecting. Assume that $|\hg|> \binom{n}{\ell-s}$. Then
\begin{align}\label{ineq-4.0}
|\hf| < \binom{s-1}{t}\binom{n-s-1}{k-t}+2^s\binom{n-t-1}{k-t-1}.
\end{align}
Moreover, if $n\geq s(k-t)$, then
\begin{align}\label{ineq-4.00}
|\hf| <  \binom{s-1}{t}\binom{n-s}{k-t} +\left(\frac{2}{s-1}\binom{s}{t-1}+\binom{s-1}{t-1}+2\binom{s}{t+1}\right)\binom{n-s}{k-t-1}.
\end{align}
\end{prop}
\begin{proof}
Assume the contrary. Without loss of generality, we can suppose that $\hf$ and $\hg$ are initial ($S_{ij}$ does not change $|\hf|, |\hg|$). Since $|\hg|>\binom{n}{\ell-s}$, we infer
\[
(1,2,\ldots,s-1,s+1,s+3,\ldots ) =:G_0\in \hg.
\]
Let $T_0\in \binom{[s-1]}{t-1}$. Then by cross $t$-intersecting property
\[
T_0\cup (s,s+2,s+4,\ldots)\notin \hf.
\]
Define $T=T_0\cup\{s\}$. Then $\hf(T,[s])\subset\binom{[s+1,n]}{k-t}$ and it is pseudo intersecting, i.e.,
\[
(s+2,s+4,\ldots,s+2(k-t))\notin \hf(T,[s]).
\]
Thus
\begin{align}\label{ineq-4.1}
|\hf(T,[s])|\leq \binom{n-s}{k-t-1}.
\end{align}
For $R\subset [s]$,
\begin{align}\label{ineq-4.2}
|\hf(R,[s])|\leq \binom{n-s}{k-|R|}.
\end{align}
We use it for $R$ with $|R|>t$ and $R=t$ but $s\notin R$.

\begin{claim}\label{claim-9}
For $R\subset[s]$ with $|R|=t-i$ and $i\geq 1$, $\hf(R,[s])$ is pseudo $(2i+1)$-intersecting.
\end{claim}
\begin{proof}
Let
\[
\tilde{R}:= (s+1,s+2,\ldots,s+2i-1,s+2i,s+2i+2,\ldots,s+2(k-t)).
\]
Set $Q= [t-1]\cup (s,s+2,\ldots,s+2(k-t))$ and note $|Q\cap G_0|=t-1$ whence $Q\notin \hf$. Since $Q\prec R\cup \tilde{R}$, $R\cup \tilde{R}\notin \hf$, i.e., $\tilde{R}\notin \hf(R,[s])$. Thus $\hf(R,[s])$ is pseudo $(2i+1)$-intersecting.
\end{proof}

 For $|R|=t-i$ with $1\leq i\leq t$, by Claim \ref{claim-9}
\[
|\hf(R,[s])|\leq \binom{n-s}{k-(t-i)-2i-1} =\binom{n-s}{k-t-i-1}.
\]
Now
\begin{align}\label{ineq-4.9}
|\hf|&=\sum_{R\subset [s]} |\hf(R,[s])|\nonumber\\[5pt]
&=\sum_{R\subset [s],|R|\leq t-1} |\hf(R,[s])|+\sum_{R\in  \binom{[s]}{t}} |\hf(R,[s])|+\sum_{R\subset [s],|R|\geq  t+1} |\hf(R,[s])|\nonumber\\[5pt]
&\leq \sum_{0\leq i\leq t-1}\binom{s}{i}\binom{n-s}{k-2t+i-1}+\binom{s-1}{t}\binom{n-s}{k-t}
+\binom{s-1}{t-1}\binom{n-s}{k-t-1}\nonumber\\[5pt]
& \qquad + \sum_{t+1\leq i\leq s}\binom{s}{i} \binom{n-s}{k-i}.
\end{align}
Using $\binom{n-s}{k-2t+i-1} < \binom{n-s}{k-t-1}$ for $i\leq t-1$ and $\binom{n-s}{k-i}\leq \binom{n-s}{k-t-1}$ for $i\geq t+1$, we conclude that
\begin{align*}
|\hf|&< \binom{s-1}{t}\binom{n-s}{k-t}+ \sum_{0\leq i\leq s}\binom{s}{i} \binom{n-s}{k-t-1}-\binom{s-1}{t}\binom{n-s}{k-t-1}\\[5pt]
&= \binom{s-1}{t}\binom{n-s-1}{k-t}+ 2^s \binom{n-s}{k-t-1}.
\end{align*}
This proves \eqref{ineq-4.0}.

If $n\geq s(k-t)$ then for $1\leq i\leq t-1$
\begin{align*}
\frac{\binom{s}{i}\binom{n-s}{k-2t+i-1}}{\binom{s}{i-1}\binom{n-s}{k-2t+i-2}} &=\frac{(s-i+1)(n-s-k+2t-i+2)}{i(k-2t+i-1)} \\[5pt]
&\geq \frac{(s-t+2)(n-s-k+t+3)}{(t-1)(k-t-2)}\\[5pt]
&> \frac{(s-t+2)(s-1)(k-t-1)}{(t-1)(k-t-2)}\\[5pt]
&> 2
\end{align*}
and
\[
\frac{\binom{n-s}{k-t-2}}{\binom{n-s}{k-t-1}} =\frac{k-t-1}{n-s-k+t+2} < \frac{k-t-1}{(s-1)(k-t-1)} \leq \frac{1}{s-1}.
\]
It follows that
\begin{align}\label{ineq-4.7}
\sum_{0\leq i\leq t-1}\binom{s}{i}\binom{n-s}{k-2t+i-1} &< \binom{s}{t-1}\binom{n-s}{k-t-2}\sum_{i=0}^\infty 2^{-i}\nonumber\\[5pt]
&=2\binom{s}{t-1}\binom{n-s}{k-t-2}\nonumber\\[5pt]
&<\frac{2}{s-1}\binom{s}{t-1}\binom{n-s}{k-t-1}.
\end{align}

For $t+1\leq i\leq s-1$,
\begin{align*}
\frac{\binom{s}{i+1} \binom{n-s}{k-i-1}}{\binom{s}{i} \binom{n-s}{k-i}}&=\frac{(s-i)(k-i)}{(i+1)(n-s-k+i+1)}\\[5pt]
&\leq \frac{(s-t-1)(k-t-1)}{(t+2)(n-s-k+t+2)}\\[5pt]
&< \frac{(s-t-1)(k-t-1)}{(t+2)(s-1)(k-t-1)}\\[5pt]
&< \frac{1}{2}.
\end{align*}
It follows that
\begin{align}\label{ineq-4.8}
\sum_{t+1\leq i\leq s}\binom{s}{i} \binom{n-s}{k-i} < \binom{s}{t+1}\binom{n-s}{k-t-1}\sum_{i=0}^\infty 2^{-i}=2\binom{s}{t+1}\binom{n-s}{k-t-1}.
\end{align}
Combining \eqref{ineq-4.9}, \eqref{ineq-4.7} and \eqref{ineq-4.8}, we conclude that
\begin{align*}
|\hf| &< \frac{2}{s-1}\binom{s}{t-1}\binom{n-s}{k-t-1}+\binom{s-1}{t}\binom{n-s}{k-t}
+\binom{s-1}{t-1}\binom{n-s}{k-t-1}\\[5pt]
&\qquad +2\binom{s}{t+1}\binom{n-s}{k-t-1}\\[5pt]
&= \binom{s-1}{t}\binom{n-s}{k-t} +\left(\frac{2}{s-1}\binom{s}{t-1}+\binom{s-1}{t-1}+2\binom{s}{t+1}\right)\binom{n-s}{k-t-1}.
\end{align*}
\end{proof}

Consider the obvious construction:
\[
\hg= \left\{G\in \binom{[n]}{\ell}\colon [s]\subset G\right\}, \ \hf=\left\{F\in \binom{[n]}{k}\colon |F\cap [s]|\geq t\right\}.
\]
Then $(\hf,\hg)$ are cross $t$-intersecting and
\[
|\hg| =\binom{n-s}{\ell-s},\ |\hf| = \binom{s}{t}\binom{n-s}{k-t}+\sum_{t<j\leq s}\binom{s}{j}\binom{n-s}{k-j},
\]
showing that \eqref{ineq-4.0} does not hold for $|\hg|\leq \binom{n-s}{\ell-s}$.

\begin{cor}
Let $\hf\subset \binom{[n]}{k}$ be $t$-intersecting with $n\geq (t+2)(k-t)$ and  $|\hf|> (t+1)\binom{n-1}{k-t-1}$. If $\varrho(\hf)<\frac{t}{t+1}$, then for every $P\in \binom{[n]}{2}$,
\begin{align}\label{ineq-4.3}
|\hf(P)|\leq (t+1)\binom{n-t-2}{k-t-2}+\frac{5t^2+19t+24}{6}\binom{n-t-3}{k-t-3}.
\end{align}
\end{cor}
\begin{proof}
If there exists $\{x,y\}\subset [n]$ such that
\begin{align*}
|\hf(x,y)| &>(t+1)\binom{n-t-2}{k-t-2}+\frac{5t^2+19t+24}{6}\binom{n-t-3}{k-t-3}\\[5pt]
&>(t+1)\binom{n-t-4}{k-t-2}+\left(\frac{2}{t+1}\binom{t+2}{t-1}+\binom{t+1}{t-1}+2\binom{t+2}{t+1}\right)
\binom{n-t-4}{k-t-3},
\end{align*}
note that $\hf(x,y)\subset \binom{[n]\setminus \{x,y\}}{k-2}$, $\hf(\bar{x},\bar{y})\subset \binom{[n]\setminus \{x,y\}}{k}$ are cross $t$-intersecting, by applying Proposition \ref{thm-4.3} with $s=t+2$ we infer
\[
|\hf(\bar{x},\bar{y})|\leq \binom{n-2}{k-t-2}.
\]
Since $\varrho(\hf)<\frac{t}{t+1}$ implies
\[
|\hf(\bar{x})|, |\hf(\bar{y})|> \frac{1}{t+1}|\hf| > \binom{n-1}{k-t-1},
\]
it follows that
\[
\hf(\bar{x},y)\geq |\hf(\bar{x})| -|\hf(\bar{x},\bar{y})| > \binom{n-2}{k-t-1}
\]
and
\[
\hf(x,\bar{y})\geq |\hf(\bar{y})| -|\hf(\bar{x},\bar{y})| > \binom{n-2}{k-t-1}.
\]
But $\hf(\bar{x},y)$, $\hf(x,\bar{y})$ are cross $t$-intersecting. This contradicts Corollary  \ref{thm-tintersecting}.
\end{proof}

\begin{lem}\label{lem-4.6}
Let $\hf\subset \binom{[n]}{k}$ be initial, $t$-intersecting, $n\geq 2(t+1)(k-t)$ and $|\hf|\geq 2t(t+1)(t+2)\binom{n-t-4}{k-t-2}$ then
\[
\varrho(\hf)>\frac{t}{t+1}.
\]
\end{lem}

\begin{proof}
Consider a subset $P\subset [t+1]$, $|P|\leq t-1$.

\begin{claim}\label{claim-10}
$\hf(P,[t+1])$ is $1+2(t-|P|)$-intersecting.
\end{claim}
\begin{proof}
Suppose for contradiction that $\bar{F}_1, \bar{F}_2\in \hf(P,[t+1])$ satisfy $\bar{F}_1\cap \bar{F}_2=D$ with  $|D|\leq 2(t-|P|)$. Since $\hf$ is $t$-intersecting, we infer $|D|\geq t-|P|$. If $|D|=t-|P|$ then choose $y\in D$ and $x\in [t+1]\setminus P$ and set $F_1=\bar{F}_1\cup P$, $F_2=(\bar{F}_2\cup P\cup \{x\})\setminus \{y\}$. By initiality $F_2\prec \bar{F}_2\cup P$ implies $F_2\in \hf$. But $|F_1\cap F_2|=|P|+|D|-1=t-1$, a contradiction.

 If $|D|\geq t+1-|P|$, then choose $E\subset D$, $|E|=t+1-|P|$ and set $F_1=\bar{F}_1\cup P$, $F_2=(\bar{F}_2\cup [t+1])\setminus E$. Then $F_1\cap F_2=P\cup D\setminus E$ whence
\[
|F_1\cap F_2|=|P|+|D|-|E| \leq |P|+2t-2|P|-(t+1-|P|)=t-1.
\]
Now $F_1\in \hf$ and $\bar{F}_2\cup P\in \hf$ by definition and $F_2\prec \bar{F}_2\cup P$. Hence $F_2\in \hf$ contradicting the $t$-intersecting property.
\end{proof}

Define $\hf_i=\hf([t+1]\setminus \{i\},[t+1])$. By initiality
\[
\hf_1\subset\hf_2\subset \ldots\subset \hf_{t+1}.
\]
Since $\hf(\bar{1})$ is $(t+1)$-intersecting, $\hf_1$  is intersecting. Thus by \eqref{ineq-ishadow} $|\partial \hf_1|\geq |\hf_1|$. By initiality $\partial \hf_1\subset \hf_0:=\hf([t+1])$. We infer that
\begin{align}\label{ineq-4.5}
|\hf_0|+|\hf_1|+|\hf_2|+\ldots+|\hf_{t+1}|\geq (t+2)|\hf_1|.
\end{align}
For any $P\in \binom{[2,t+1]}{t-j}$, by Claim \ref{claim-10} we know $\hf(P,[t+1])$ is $(2j+1)$-intersecting. Note that $n\geq 2(t+1)(k-t)> (2j+2)(k-t-j)$ for $j=1,\ldots,t$. By \eqref{ineq-ekr}, we infer
\[
|\hf(P,[t+1])| \leq \binom{n-t-1-2j-1}{k-(t-j)-2j-1}= \binom{n-t-2-2j}{k-t-1-j}.
\]
Note that for $j\geq  k-t$ we simply have $|\hf(P,[t+1])|=0$.
Thus,
\begin{align*}
|\hf(\bar{1})|=&\sum_{P\subset [2,t+1]} |\hf(P,[t+1])| \\[5pt]
=&|\hf_1| +\sum_{0\leq i\leq t-1}\sum_{P\in \binom{[2,t+1]}{i}} |\hf(P,[t+1])|\\[5pt]
\leq& |\hf_1| +\sum_{1\leq j\leq \min\{t,k-t-1\}} \binom{t}{t-j}\binom{n-t-2-2j}{k-t-1-j}.
\end{align*}
For $k=t+2$, we have
\[
|\hf(\bar{1})|=|\hf_1|+t\binom{n-t-4}{k-t-2}.
\]
For $2\leq j\leq \min\{t,k-t-1\}$
\begin{align*}
\frac{\binom{t}{j}\binom{n-t-2-2j}{k-t-1-j}}{\binom{t}{j-1}\binom{n-t-2j}{k-t-j}} =\frac{(t-j+1)(k-t-j)(n-k-j)}{j(n-t-2j)(n-t-2j-1)}\leq \frac{(t-1)(k-t-2)(n-k-2)}{2(n-3t)(n-3t-1)}.
\end{align*}
Since $n\geq 2(t+1)(k-t)$ and $k\geq t+3$ implies that
\[
\frac{n-k-2}{n-3t} < 2,\ \frac{(t-1)(k-t-2)}{n-3t-1}< \frac{1}{2},
\]
it follows that $\binom{t}{j}\binom{n-t-2-2j}{k-t-1-j}<\frac{1}{2}\binom{t}{j-1}\binom{n-t-2j}{k-t-j}$.
Thus,
\begin{align}\label{ineq-4.6}
|\hf(\bar{1})|& \leq |\hf_1| +\sum_{1\leq j\leq t} \binom{t}{t-j}\binom{n-t-2-2j}{k-t-1-j}\nonumber\\[5pt]
& < |\hf_1| +t\binom{n-t-4}{k-t-2} \sum_{i=0}^{\infty} 2^{-i}\nonumber\\[5pt]
&= |\hf_1| +2t\binom{n-t-4}{k-t-2}.
\end{align}
By \eqref{ineq-4.5} and \eqref{ineq-4.6}, we obtain that
\begin{align*}
|\hf(\bar{1})| &\leq
\frac{1}{t+2} \sum_{0\leq i\leq t+1}|\hf_i| +2t\binom{n-t-4}{k-t-2}\\[5pt]
&<
\frac{1}{t+2} |\hf|+2t\binom{n-t-4}{k-t-2}\\[5pt]
&\leq \frac{1}{t+2} |\hf|+ \frac{1}{(t+1)(t+2)} |\hf|\\[5pt]
&= \frac{1}{t+1}|\hf|,
\end{align*}
and the lemma follows.
\end{proof}

\begin{proof}[Proof of Theorem \ref{thm-main5}]
Suppose to the contrary that $|\hf|> (t+1)\binom{n-1}{k-t-1}$ and $\varrho(\hf)\leq  \frac{t}{t+1}$. Since $n\geq 2t(t+2)k$, we infer
 \begin{align}\label{ineq-4.10}
 |\hf|\geq (t+1)\binom{n-1}{k-t-1} &> (t+1)2t(t+2)\binom{n-2}{k-t-2}> 4(t+1)(t+2)\binom{n-t-2}{k-t-2}.
 \end{align}
 Shift $\hf$ ad extremis for $\varrho(\hf)\leq  \frac{t}{t+1}$ and let $\mathds{H}$ be the graph formed by the shift-resistant pairs. For every $P\in \binom{[n]}{2}$, by \eqref{ineq-4.3} and $n\geq 2t(t+2)k>\frac{5t^2+19t+24}{6}k$ we infer
\begin{align}\label{ineq-new4.3}
|\hf(P)| < (t+1)\binom{n-t-2}{k-t-2}+\frac{5t^2+19t+24}{6}\binom{n-t-3}{k-t-3}< (t+2)\binom{n-t-2}{k-t-2}.
\end{align}

\begin{claim}
$\mathds{H}$ has matching number at most 1.
\end{claim}
\begin{proof}
Let $(a_1,b_1),(a_2,b_2)\in \mathds{H}$. Set $\hg_i = \hf_{\{a_i,b_i\}}$, $i=1,2$. Since $\varrho(S_{a_ib_i}(\hf))>\frac{t}{t+1}|\hf|$, we infer $|\hg_i|\geq \frac{t}{t+1}|\hf|$. By \eqref{ineq-new4.3} we have
\[
|\hg_1\cap\hg_2|  <  4(t+2)\binom{n-t-2}{k-t-2}.
\]
It follows that
\begin{align*}
|\hf| &\geq |\hg_1|+|\hg_2|-|\hg_1\cap \hg_2|\\[5pt]
&> \frac{2t}{t+1}|\hf| - 4(t+2)\binom{n-t-2}{k-t-2}\\[5pt]
&\overset{\eqref{ineq-4.10}}>  \frac{2t}{t+1}|\hf| - \frac{1}{t+1}|\hf| \geq |\hf|,
\end{align*}
contradiction.
\end{proof}

Note that $n\geq 2t(t+2)k$ implies
\begin{align*}
|\hf|\geq (t+1)\binom{n-1}{k-t-1}>2t(t+1)(t+2)\binom{n-t-4}{k-t-2}.
\end{align*}
By Lemma \ref{lem-4.6}, we may assume that $\mathds{H}$ has matching number one. For convenience assume that $(n-1,n)\in \mathds{H}$. Let
\[
\ha=\left\{A\in \binom{[n-2]}{k-1}\colon A\cup \{x\}\in \hf, \ x=n-1 \mbox{ or } x=n\right\},\ \hb=\binom{[n-2]}{k}\cap \hf.
\]
Since
\[
|\hf(n-1,n)|+|\hf(\overline{n-1},n)\cup\hf(n-1,\overline{n})|\geq \frac{t}{t+1}|\hf|,
\]
by \eqref{ineq-4.3} and $t\geq 2$ we infer
\begin{align}\label{ineq-4.4}
\ha(\bar{1},\bar{2})&\geq \frac{t}{t+1}|\hf| -|\hf(n-1,n)|-\sum_{i\in \{1,2\},j\in \{n-1,n\} }|\hf(i,j)|\nonumber\\[5pt]
&\geq \frac{t}{t+1}|\hf| -5(t+2)\binom{n-t-2}{k-t-2}\nonumber\\[5pt]
&\overset{\eqref{ineq-4.10}}{\geq} 4t(t+2)\binom{n-t-2}{k-t-2} -5(t+2)\binom{n-t-2}{k-t-2}\nonumber\\[5pt]
&\overset{\eqref{ineq-key11}}{\geq}3(t+2)\cdot \frac{1}{2}\binom{n-3}{k-t-2}\nonumber\\[5pt]
&>\binom{n-4}{k-t-2}.
\end{align}

 Fix $S\subset [2]$ with $|S|\leq 1$. Since $\ha,\hb$ are initial and cross $t$-intersecting, we infer that $\ha(\bar{1},\bar{2})$ and $\hb(S,[2])$ are cross $(t+2-|S|)$-intersecting. By \eqref{ineq-4.4} we know that $\ha(\bar{1},\bar{2})$ is not pseudo $(t+2-|S|)$-intersecting. By Proposition \ref{fk2021} we infer that $\hb(S,[2])$ is pseudo $(t+3-|S|)$-intersecting. Therefore,
 \[
 |\hb(S,[2])|\leq \binom{n-4}{k-|S|-(t+3-|S|)} = \binom{n-4}{k-t-3}.
 \]
Note that \eqref{ineq-new4.3} implies $\hb([2])< (t+2)\binom{n-t-2}{k-t-2}$. Thus,
\begin{align*}
|\hb|&= \sum_{S\subset [2]}  |\hb(S,[2])|\\[5pt]
& < 3\binom{n-4}{k-t-3}+(t+2)\binom{n-t-2}{k-t-2}\\[5pt]
& < \frac{3(k-t-2)}{n-3}\binom{n-3}{k-t-2}+(t+2)\binom{n-t-2}{k-t-2}\\[5pt]
 &\overset {\eqref{ineq-key11}}{\leq}\frac{6(k-t-2)}{n-3} \binom{n-t-2}{k-t-2}+(t+2)\binom{n-t-2}{k-t-2}\\[5pt]
&< (t+3)\binom{n-t-2}{k-t-2}.
\end{align*}
Then $\varrho(\hf)\leq \frac{t}{t+1}$ implies
\[
|\hf(\overline{n-1},n)|\geq |\hf(\overline{n-1})|-|\hb|>\frac{1}{t+1}|\hf|-(t+3)\binom{n-t-2}{k-t-2}\overset{\eqref{ineq-4.10}}{>} (3t+5)\binom{n-t-2}{k-t-2}
\]
and
\[
|\hf(n-1,\overline{n})|\geq |\hf(\bar{n})|-|\hb|>\frac{1}{t+1}|\hf|-(t+3)\binom{n-t-2}{k-t-2}\overset{\eqref{ineq-4.10}}{>} (3t+5)\binom{n-t-2}{k-t-2}.
\]
Now by \eqref{ineq-new4.3}
\begin{align*}
|\hf(\overline{\{1,n-1\}},n)|\geq |\hf(\overline{n-1},n)|-|\hf(1,n)|> (2t+3)\binom{n-t-2}{k-t-2}
\end{align*}
and
\begin{align*}
|\hf(\overline{\{1,n\}},n-1)|\geq |\hf(n-1,\overline{n})|-|\hf(1,n-1)|> (2t+3)\binom{n-t-2}{k-t-2}.
\end{align*}
By \eqref{ineq-key11},
\[
(2t+3)\binom{n-t-2}{k-t-2} \geq \frac{(2t+3)}{2} \binom{n-3}{k-t-2} > \binom{n-3}{k-t-2},
\]
this contradicts the fact that $\hf(\overline{\{1,n-1\}},n), \hf(\overline{\{1,n\}},n-1)\subset \binom{[2,n-2]}{k-1}$ are cross $(t+1)$-intersecting. Thus the theorem holds.
\end{proof}

\section{Maximum degree results for $r$-wise intersecting families}

A family $\hf\subset\binom{[n]}{k}$ is called {\it $r$-wise $t$-intersecting} if $|F_1\cap \ldots \cap F_r|\geq t$ for all $F_1,\ldots,F_r\in \hf$.

The investigation of $r$-wise $t$-intersecting families has a long history  (\cite{BD}, \cite{F76-2}, \cite{F79}, \cite{F91-2}, \cite{F19}, \cite{FT11}). Many of those results concern general (non-uniform) families.
One of the early gems of extremal set theory is the following.

\begin{thm}[Brace-Daykin Theorem \cite{BD}]
Suppose that $\hf\subset 2^{[n]}$ is non-trivial $r$-wise intersecting, $r\geq 3$. Then
\begin{align}\label{ineq-5.1}
|\hf| \leq \frac{r+2}{2^{r+1}}2^n
\end{align}
with equality iff $\hf$ is isomorphic to
\[
\hb(n,r) =\{B\subset [n]\colon |B\cap [r+1]|\geq r\}.
\]
\end{thm}

For $r$-wise intersecting uniform families, $\hf\subset \binom{[n]}{k}$, $n\geq 2k$,
\begin{align}\label{ineq-5.2}
|\hf| \leq \binom{n-1}{k-1}
\end{align}
follows from the Erd\H{o}s-Ko-Rado Theorem. In \cite{F76-2} it is proved for the best possible range $n\geq \frac{r}{r-1}k$. In \cite{F87} it is shown that unlike the case $r=2$, even if $n=\frac{r}{r-1}k$, the full star is the unique family attaining equality in \eqref{ineq-5.2}.

Let us introduce the notation
\[
t_j(\hf)=\max\left\{t\colon \hf \mbox{ is  $j$-wise $t$-intersecting}\right\}.
\]
Then $\hf$ is $r$-wise $t$-intersecting iff $t_r(\hf)\geq t$.

Note that $t_j(\hf)=t_{j+1}(\hf)>0$ implies that $\hf$ is trivial. Indeed, pick $F_1,\ldots,F_j\in \hf$ satisfying $F_1\cap \ldots\cap F_j=:T$ satisfying $|T|=t_j(\hf)$. Then $T\subset F$ for all $F\in \hf$.

Similarly, if $t_j(\hf)=1+t_{j+1}(\hf)$ then $|F\cap T|\geq |T|-1$ holds for the above $T$ and every $F\in \hf$. These considerations lead to
\begin{lem}
Suppose that $\hf$ is non-trivial $r$-wise $t$-intersecting, $r\geq 3$. Then
\begin{align}\label{ineq-5.3}
t_2(\hf)\geq t+r-2.
\end{align}
Moreover, in case of equality $\tau_{t+r-3}(\hf)=t+r-2$.
\end{lem}
\begin{proof}
Since $\hf$ is non-trivial, $t_j(\hf)\geq t_{j+1}(\hf)+1$ for $2\leq j<r$. This implies \eqref{ineq-5.3}. Suppose now that equality holds and choose $T\in \binom{[n]}{t+r-2}$ with $T=F_1\cap F_2$ for some $F_1,F_2\in\hf$. If $|F_3\cap T|\leq t+r-4$ for some $F_3\in \hf$ then using non-triviality we can find $F_1,\ldots,F_r$ to satisfy $|F_1\cap \ldots \cap F_r|\leq t-1$, a contradiction.

Hence $|F\cap T|\geq |T|-1$ for all $F\in \hf$ proving that $T$ is a $(t+r-3)$-cover. If $\tau_{t+r-3}(\hf) =t+r-3$ then $\hf$ is trivial, thus $\tau_{t+r-3}=t+r-2$.
\end{proof}

Let us prove an inequality concerning $\varrho(\hf)$ for $r$-wise intersecting uniform families.

\begin{prop}
Let $\hf\subset\binom{[n]}{k}$ be non-trivial, saturated $r$-wise intersecting with $n\geq \frac{r}{r-1}k$. If $\tau_{r-1}(\hf)\leq r$, then $\varrho(\hf)>\frac{r}{r+1}$.
\end{prop}

\begin{proof}
By \eqref{ineq-5.3} $\hf$ is two-wise $(r-1)$-intersecting. Let $R$ be an $(r-1)$-transversal of $\hf$ with $|R|\leq r$. If $|R|=r-1$ then $\hf$ is trivial $(r-1)$-intersecting. Thus we may assume that $|R|=r$.   Note that for all $F_1,\ldots,F_{r-1}\in \hf$,
\[
|F_1\cap \ldots\cap F_{r-1}\cap R|\geq r-(r-1)=1.
 \]
 Hence $\hf_0:=\hf(R)=\binom{[n]\setminus R}{k-r}$ by saturatedness. Define
\[
\hf_x = \{F\setminus R\colon F\in \hf, F\cap R=R\setminus \{x\}\} \mbox{ for } x\in R.
\]
Now $|\hf|=|\hf_0|+\sum\limits_{x\in R}|\hf_x|$. Obviously, $\{\hf_x\colon x\in R\}$ are $r$-wise cross-intersecting. By a result in \cite{FT11}
\[
\prod_{x\in R} |\hf_x| \leq \binom{n-r-1}{k-r}^r.
\]
Let us choose $z$ to minimize $|\hf_z|$. Then
\[
|\hf_z| \leq \binom{n-r-1}{k-r} <\binom{n-r}{k-r}=|\hf_0|.
\]
Thus we conclude that
\begin{align*}
\varrho(\hf)\geq \frac{|\hf(z)|}{|\hf|} =\frac{|\hf|-|\hf_z|}{|\hf|} =1-\frac{|\hf_z|}{|\hf_0| +\sum\limits_{x\in R}|\hf_x|} > 1-\frac{|\hf_z|}{(r+1)|\hf_z|} =\frac{r}{r+1}.
\end{align*}
\end{proof}

{\noindent\bf Remark.} Since $\hf$ is $(r-1)$-intersecting, $\hf_x,\hf_y$ are cross-intersecting. For $n\geq 2k$ we could use the product version of the Erd\H{o}s-Ko-Rado Theorem ($t=1$ case) due to Pyber \cite{Pyber86} . However, to cover the range $2k>n\geq \frac{r}{r-1}k$ as well we need \cite{FT11}, the product version of \cite{F76-2}.

\end{document}